\documentclass[12pt,reqno]{amsart}
\usepackage{mathrsfs}
\usepackage{amsgen}
\usepackage{amscd}
\usepackage{amsmath}
\usepackage{latexsym}
\usepackage{amsfonts}
\usepackage{amssymb}
\usepackage{amsthm}
\usepackage{graphicx}
\usepackage{color}
\usepackage{amsthm,amssymb}
\usepackage{graphicx}
\usepackage{enumerate}
\usepackage{amssymb,amsmath,graphicx,amsfonts,euscript}
\usepackage{color}
\usepackage{mathrsfs}
\usepackage{empheq}
\usepackage{cases}

\usepackage{hyperref}

\usepackage[margin=3cm, a4paper]{geometry}

\def\H{{\cal H}}

\def\R{\mathbb{R}}

\def\H2{H^2(\R^N)}
\def\L2{L^2(\R^N)}
\def\to{\rightarrow}

\newcommand{\dx}{\,\mathrm{d}x}

\def\H{{\cal H}}

\def\H1{H^1(\R)}

 \newcommand{\Del}[1]{}

\numberwithin{equation}{section}

\newtheorem{thm}{Theorem}[section]
\newtheorem{cor}[thm]{Corollary}
\newtheorem{lem}[thm]{Lemma}
\newtheorem{prop}[thm]{Proposition}
\newtheorem{definition}[thm]{Definition}

\theoremstyle{remark}

\newtheorem*{exam*}{Examples}
\allowdisplaybreaks

\begin{document}

\setcounter{page}{1}

\title[Instability of solitary waves]{Instability of the solitary waves for the generalized Boussinesq equations}

\author{Bing Li}
\address{Center for Applied Mathematics\\
Tianjin University\\
Tianjin 300072, China}
\email{binglimath@gmail.com}
\thanks{Funding: The work of the first and third authors was partially supported by National
Natural Science Foundation of China grants 11771325 and 11571118.
The work of the second author
was partially supported by JSPS KAKENHI grant 15K04968.
The work of the third author was also partially supported by National Youth Topnotch Talent Support Program in China.
The work of the fourth author
was partially supported by Research Council of Norway (No. 250070).}

\author{Masahito Ohta}
\address{Department of Mathematics\\
Tokyo University of Science\\
Tokyo 162-8601, Japan}
\email{mohta@rs.tus.ac.jp}
\thanks{}

\author{Yifei Wu}
\address{Center for Applied Mathematics\\
Tianjin University\\
Tianjin 300072, China}
\email{yerfmath@gmail.com}
\thanks{* Corresponding author}

\author{Jun Xue*}
\address{Department of Mathematical Sciences\\
Norwegian University of Science and Technology\\
Trondheim 7491, Norway}
\email{jxuemath@hotmail.com}
\thanks{\href{https://doi.org/10.1137/18M1199198}{SIAM Journal on Mathematical Analysis, Volume 52, Issue 4, Pages 3192--3221}. This paper is the final version.}




\maketitle

\begin{abstract}\noindent
In this work, we consider the generalized Boussinesq equation
\begin{align*}
\partial_{t}^2u-\partial_{x}^2u+\partial_{x}^2(\partial_{x}^2u+|u|^{p}u)=0,\qquad (t,x)\in\R\times \R,
\end{align*}
with $0<p<\infty$. This equation has the traveling wave solutions $\phi_\omega(x-\omega t)$, with the frequency $\omega\in (-1,1)$ and $\phi_\omega$ satisfying
\begin{align*}
-\partial_{xx}{\phi}_{\omega}+(1-{\omega^2}){\phi}_{\omega}-{\phi}_{\omega}^{p+1}=0.
\end{align*}
Bona and Sachs \cite{bosach-CMP-88} proved that the traveling wave $\phi_\omega(x-\omega t)$ is orbitally stable when $0<p<4,$ $\frac p4<\omega^2<1$. Liu \cite{Liu-JDDE-93} proved the orbital instability under the conditions $0<p<4,$ $\omega^2<\frac p4$ or $p\ge 4,$ $\omega^2<1$. In this paper, we prove the orbital instability  in the degenerate case $0<p<4,\omega^2=\frac p4$ .
\end{abstract}

\section{Introduction}
\vskip 0.2cm
In this paper, we consider the stability theory of the generalized Boussinesq equation
\begin{equation}\label{eq:GBQ}
\begin{split}
&\partial_{t}^2u-\partial_{x}^2u+\partial_{x}^2(\partial_{x}^2u+|u|^{p}u)=0,\qquad (t,x)\in\R\times \R,
\end{split}
\end{equation}
with the initial data
\begin{align}\label{eq:initialdata}
u(0,x) =u_0(x),\quad u_t(0,x) =u_1(x).
\end{align}
Here $0<p<\infty$.

The Boussinesq equation was originally derived by Boussinesq \cite{bous-JMPA-1872}. It arises from studying an approximation
to the evolution of the free surface of a water wave. 

Equation \eqref{eq:GBQ} has the solitary wave solution $u(x,t)=\phi_\omega(x-\omega t)$, where $\phi_\omega$ is the ground state solution of the following elliptic equation
\begin{align}\label{elliptic}
-\partial_{xx}{\phi}_{\omega}+(1-{\omega^2}){\phi}_{\omega}-{\phi}_{\omega}^{p+1}=0, \qquad |\omega|< 1.
\end{align}
The ground state solution $\phi_\omega$ is an even function and it has the property of exponential decay, that is, $|\phi_\omega|\leq C_1e^{-C_2|x|}$ for some $C_1,C_2>0$ and $|\partial_x\phi_\omega|\leq C_3e^{-C_4|x|}$ for some $C_3,C_4>0$.

Equation (\ref{eq:GBQ}) has the equivalent system form
 \begin{equation}\label{eq:ut}
   \left\{ \aligned
    &u_t=v_x,
    \\
    &v_t=(-u_{xx}+u-|u|^pu)_x.
   \endaligned
  \right.
 \end{equation}
Then the system \eqref{eq:ut} has the following solitary wave solution
\begin{equation*}
\left(\begin{array}{c}
u\\
v\\
\end{array}\right)
(t,x)=\left(\begin{array}{c}
{\phi}_{\omega}(x-{\omega t})\\
-{\omega}{\phi}_{\omega}(x-{\omega t})\\
\end{array}\right).
\end{equation*}
For the $H^1\times L^2$-solution $(u,v)^T$ of (\ref{eq:GBQ})--(\ref{eq:initialdata}), the momentum Q and the energy E are conserved under the flow, where
\begin{align}
Q\left(\begin{array}{c}
u\\
v\\
\end{array}\right)&=\int_{\R} uv\dx;\label{Momentum}\\
E\left(\begin{array}{c}
u\\
v\\
\end{array}\right)&=\frac{1}{2}\int_{\R} (|u_x|^2+|u|^2+|v|^2)\dx
-\frac{1}{p+2}\int_{\R} |u|^{p+2}\dx.\label{Energy}
\end{align}

There are several related results for the generalized Boussinesq equation. For a local existence result, Liu \cite{Liu-JDDE-93} proved that the system \eqref{eq:ut} is locally well-posed in $H^1(\mathbb R)\times L^2(\mathbb R)$. For the stability theories, Bona and Sachs \cite{bosach-CMP-88} proved that when $0<p<4,$ $\frac p4<\omega^2<1,$ the solitary wave solution is orbitally stable. Liu \cite{Liu-JDDE-93} proved the orbital instability if $0<p<4$ and $\omega^2<\frac p4$ or $p\ge 4$ and $\omega^2<1$. Liu \cite{Liu-SJMA-95} proved that when the wave speed $\omega=0$, the solitary wave solution is strongly unstable by blow-up.
Liu, Ohta, and Todorova \cite{LiuOhtaTod-AIPANL-07} showed that when $0<p<\infty$ and $0 < 2(p+2)\omega^2 <p,$ the solitary wave solution is strongly unstable by blow-up. For the abstract Hamiltonian systems, we refer the readers to Grillakis, Shatah, and Strauss \cite{GrShStr-87,GrShStr-90} for the stability/instability theories, in which the Vakhitov-Kolokolov stability criteria of the solitary waves were confirmed except the degenerate cases. In the degenerate cases, it was also proved by Comech and Pelinovsky \cite{CoPe-CPAM-03} (see also \cite{Ohta-JFA-11}) that the solitary wave solution is orbitally unstable under some regularity restrictions in the nonlinearity (for example, $p$ should be suitably large in our cases). In this paper, we consider the stability theory on the solitary wave solutions of the generalized Boussinesq equation and aim to show its instability in the degenerate cases without any regularity restriction.
It is worth noting that none of the above two frameworks of Grillakis, Shatah and Strauss \cite{GrShStr-87, GrShStr-90} and Comech and Pelinovsky \cite{CoPe-CPAM-03} are available in our cases, either because of the degeneration or because of insufficient regularity of the nonlinearity.

Before starting our theorem, we give some definitions. Let $v_0=\int_{-\infty}^x u_1(y)\,\mathrm dy$, $\vec u=(u,v)^T$, $\vec u_0=(u_0,v_0)^T$,  and  $\overrightarrow{\Phi_\omega}=(\phi_\omega,-\omega \phi_\omega)^T$.
For $\varepsilon>0$, we denote the set $ U_\varepsilon\big(\overrightarrow{\Phi_\omega}\big)$ as
\begin{align}\label{set:UPhi}
U_\varepsilon \big(\overrightarrow{\Phi_\omega}\big)=\{\vec u\in H^1(\R)\times L^2(\R):\inf_{y\in \R}\| \vec u-\overrightarrow{\Phi_\omega}(\cdot-y)\|_{H^1\times L^2}<\varepsilon\}.
\end{align}
\begin{definition}\label{def:stability}
We say that the solitary wave solution $\phi_\omega(x-\omega t)$ of \eqref{eq:GBQ} is orbitally stable
if for any $\varepsilon >0,$ there exists $\delta >0$ such that if $\|\vec u_0-\overrightarrow{\Phi_\omega}\|_{H^1\times L^2}< \delta$,
then the solution $\vec u(t)$ of \eqref{eq:GBQ} with $\vec u(0)=\vec u_0$ exists for all $t\in \mathbb R$,
and $\vec u(t)\in U_\varepsilon\big(\overrightarrow{\Phi_\omega}\big)$ for all $t\in \mathbb R$.
Otherwise, $\phi_\omega(x-\omega t)$ is said to be orbitally unstable.
\end{definition}
Then the main result in the present paper is the following.
\begin{thm}\label{thm:main}
Let $0<p<4$, $\omega\in (-1,1)$ and $\phi_\omega$ be the solution of \eqref{elliptic}. If $|\omega|=\sqrt{\frac{p}4}$,
then the solitary wave solution $\phi_\omega(x-\omega t)$ is orbitally unstable.
\end{thm}

The main method that we use in the present paper is from \cite{Yifei-PRE}, in which the instability of the standing wave solutions of the Klein-Gordon equation in the degenerate cases was proved. Instead of construction of the Lyapunov functional, the argument in  \cite{Yifei-PRE} is to use the monotonicity of the virial quantity to control the modulations. However, the details of this argument depend sensitively on the problem, and the key ingredients of our proof are the following.

(1) The nonstandard modulation and coercivity properties are given. More precisely, define the functional $S_\omega$ as
\begin{align*}
S_\omega(\vec u)=E(\vec u)+\omega Q(\vec u).
\end{align*}
Inspired by \cite{MePa-InvenMath-04, MePa-Annal-05, Weinstein-CPAM-86}, we establish the following nonstandard coercivity properties.   We prove the existence of suitable directions $\overrightarrow{\Gamma_\omega},\overrightarrow{\Psi_\omega}\in H^1(\mathbb R)\times L^2(\mathbb R)$ such that the following coercivity properties hold. Suppose that $\vec \eta \in H^1(\R)\times L^2(\R)$ satisfies
\begin{align*}
\Big\langle\vec \eta, \overrightarrow{\Gamma_\omega}\Big\rangle
=\Big\langle\vec \eta ,\overrightarrow{\Psi_\omega}\Big\rangle=0;
\end{align*}
then
\begin{align*}
\Big\langle S_\omega''\big(\overrightarrow{\Phi_\omega}\big)\vec \eta, \vec \eta\Big\rangle
\gtrsim \big\|\vec \eta\big\|_{H^1\times L^2}^2.
\end{align*}
The choices of  $\overrightarrow{\Gamma_\omega},\overrightarrow{\Psi_\omega}$ play important roles in our estimation. $\overrightarrow{\Psi_\omega}$ can be regarded as the negative direction, which satisfies $\big\langle S_\omega''\big(\overrightarrow{\Phi_\omega}\big)\overrightarrow{\Psi_\omega}, \overrightarrow{\Psi_\omega}\big\rangle<0$. However, we remark that $\overrightarrow{\Gamma_\omega} \notin Ker(S_{\omega}''(\overrightarrow{\Phi_\omega}))$, which is much different from the standard.
Moreover, by suitably setting the translation and scaling parameters $y,\lambda$, we can establish the modulation by writing
$$
\vec u=\Big(\vec \eta+\overrightarrow{\Phi_{\lambda(t)}}\Big)(\cdot-y(t))
$$
such that $\vec \eta$ verifies similar orthogonal conditions above (by replacing $\overrightarrow{\Gamma_\omega}, \overrightarrow{\Psi_\omega}$ with $\overrightarrow{\Gamma_\lambda}, \overrightarrow{\Psi_\lambda}$, respectively).

(2) A subtle control on the modulated translation parameters is obtained.
Instead of the rough control of the modulation parameter $y$ as $\dot y-\lambda=O(\|\vec\eta\|_{H^1\times L^2})$,
we obtain the following finer estimate:
\begin{align*}
\dot y-\lambda=\|\phi_\lambda\|_{L^2}^{-2}\Big[Q\big(\overrightarrow{\Phi_\lambda}\big)-Q\big(\overrightarrow{\Phi_\omega}\big)\Big]
-\|\phi_\lambda\|_{L^2}^{-2}\Big[Q(\vec u_0)-Q\big(\overrightarrow{\Phi_\omega}\big)\Big]+O(\|\vec\eta\|_{H^1\times L^2}^2).
\end{align*}
The subtle estimate benefits from the choices of  $\overrightarrow{\Gamma_\omega},\overrightarrow{\Psi_\omega}$ in the first step and the dynamic of the solution.
This estimate has great effects when we set up the structure of virial identity $I'(t)$ in the following.

(3) The monotonicity of the virial quantity is constructed.  The key ingredient here is to suitably define a quantity $I(t)$ and obtain its monotonicity. To this end, the crucial issue is to prove the following structure of $I'(t)$ as
\begin{align*}
I'(t)=\rho(\vec u_0)+h(\lambda)+R(\vec u),
\end{align*}
where
\begin{align*}
&\qquad\rho(\vec u_0)\geq C_1a,\quad C_1>0;\\
h(\lambda)\geq C_2&(\lambda-\omega)^2+C_3a(\lambda-\omega)^2+o(\lambda-\omega)^2,\quad C_2>0,C_3>0,
\end{align*}
and $R(\vec u)$ is an easy remainder term which can be dominated by $\rho$ and $h$. Here $a$ is the difference between the initial data and the soliton. The obstacles in the proof come from nonconservation terms among $I'(t)$ and how to eliminate the first-order terms about $\vec\eta$ and $\lambda$. These make much technical complexity.
By a delicate analysis and the utilization of the  estimates above, we overcome all difficulties and  finally  obtain the monotonicity of $I(t)$.

The rest of the paper is organized as follows. In Section \ref{sec:preliminary}, we give some preliminaries. In Section \ref{sec:coercivity}, we show the coercivity property of the Hessian $S_\omega''\big(\overrightarrow{\Phi_\omega}\big)$. In Section \ref{sec:modulation}, we show the existence of modulation parameters. In Section \ref{sec:dynamic}, we control the modulation parameters obtained in Section \ref{sec:modulation}. In Section \ref{sec:virialid}, we show the localized virial identities. Finally, we prove the main theorem in Section \ref{sec:mainthm}.
\vskip 2cm

\section{Preliminary}\label{sec:preliminary}

\vskip 0.2cm

\subsection{Notations}

For $f,g\in L^2(\mathbb{R})=L^2(\mathbb{R,R})$, we define
$$\langle f,g\rangle=\int_{\mathbb{R}}f(x)g(x)\dx$$
and regard $L^2(\mathbb{R})$ as a real Hilbert space. Similarly, for $\vec f,\vec g\in \big(L^2(\mathbb{R})\big)^2=\big(L^2(\mathbb{R,R})\big)^2$, we define
$$\langle \vec f,\vec g\rangle=\int_{\mathbb{R}}\vec f(x)^T\cdot {\vec g(x)}\dx.$$

For a function $f(x)$, its $L^{q}$-norm $\|f\|_{L^q}=\Big(\displaystyle\int_{\mathbb{R}} |f(x)|^{q}\dx\Big)^{\frac{1}{q}}$
and its $H^1$-norm $\|f\|_{H^1}=(\|f\|^2_{L^2}+\|\partial_x f\|^2_{L^2})^{\frac{1}{2}}$. For $\vec f=(f,g)^T$, its $H^1\times L^2$-norm $\|\vec f\|_{H^1\times L^2}=(\|f\|^2_{H^1}+\|g\|^2_{L^2})^{\frac 12}$.

Further, we write $X \lesssim Y$ or $Y \gtrsim X$ to indicate $X \leq CY$ for
some constant $C>0$.  We use the notation $X \sim Y$ to denote $X
\lesssim Y \lesssim X$. We also use $O(Y)$ to denote any quantity $X$ such
that $|X| \lesssim Y$ and use $o(Y)$ to denote any quantity $X$ such
that $X/Y\to 0$ if $ Y\to 0$.
Throughout the whole paper, the letter $C$ will denote various positive constants
which are of no importance in our analysis.

\subsection{Some basic definitions and properties}
In the rest of this paper, we consider the case of $0<p<4$, and $\omega_c=\sqrt{\frac p4}$, $\omega=\pm \omega_c$.
Let $\vec u=(u,v)^T$, $\overrightarrow{\Phi_\omega}=(\phi_\omega,-\omega\phi_\omega)^T$. Recall the conserved equalities,
\begin{align*}
Q(\vec u)&=\int_\mathbb R uv\dx,\\
E(\vec u)&=\frac12(\|u\|_{L^2}^2+\| u_{x}\|_{L^2}^2+\|v\|_{L^2}^2)-\frac1{p+2}\|u\|_{L^{p+2}}^{p+2}.
\end{align*}
First, we give some basic properties on the momentum and energy.

\begin{lem}\label{lem:partialQ} Let $|\omega|=\sqrt \frac p4$; then the following equality holds:
\begin{align*}
\partial_\lambda Q\big(\overrightarrow{\Phi_\lambda}\big)\Big|_{\lambda=\omega}=0.
\end{align*}
\end{lem}

\begin{proof} Note that for $\lambda \in (-1,1)$, we have
\begin{align}
Q\big(\overrightarrow{\Phi_\lambda}\big)=-\lambda\|\phi_\lambda\|_{L^2}^2.\label{Qphi}
\end{align}
By rescaling, we find
\begin{align}\label{rescaling}
\phi_\lambda(x)=(1-\lambda^2)^{\frac{1}{p}}\phi_0\left(\sqrt{1-\lambda^2}x\right).
\end{align}
This implies that
\begin{align*}
Q\big(\overrightarrow{\Phi_\lambda}\big)=-\lambda(1-\lambda^2)^{\frac{2}{p}-\frac{1}{2}}\|\phi_0\|_{L^2}^2.
\end{align*}
By a straightforward computation, we have
\begin{align*}
\partial_\lambda Q\big(\overrightarrow{\Phi_\lambda}\big)=-(1-\lambda^2)^{\frac{2}{p}-\frac{3}{2}}\Big(1-\frac{4}{p}\lambda^2\Big)\|\phi_0\|_{L^2}^2.
\end{align*}
Finally, we substitute $\lambda^2=\frac p4$ into the equality above and thus complete the proof.
\end{proof}

Now we define the functional $S_\omega$ as
\begin{align}
S_\omega(\vec u)=E(\vec u)+\omega Q(\vec u).\label{eq:Su}
\end{align}
Then we have
\begin{align}
Q'(\vec u)&=\left(\begin{array}{c}
v\\
u
\end{array}\right),\label{eq:Q'u}\\
E'(\vec u)&=\left(\begin{array}{c}
-\partial_{xx}u+u-|u|^pu\\
v
\end{array}\right),\label{eq:E'u}\\
S_{\omega}'(\vec u)&=\left(\begin{array}{c}
-u_{xx}+u-|u|^pu+{\omega}v\\
v+{\omega}u\nonumber\\
\end{array}\right).
\end{align}
Note that $S_\omega'\big(\overrightarrow{\Phi_\omega}\big)=\vec0$. Moreover, for the real-valued vector $\vec f=(f,g)^T$, a direct computation shows
\begin{align}\label{eq:S''Phi}
S_\omega''\big(\overrightarrow{\Phi_\omega}\big)\vec f=\left(
                    \begin{array}{c}
                      -\partial_{xx}f+f-(p+1)\phi_\omega^pf+\omega g \\
                      g+\omega f \\
                    \end{array}
                  \right),
\end{align}
and for any vector $\vec \xi$, $\vec \eta$,
\begin{align*}
\Big\langle S_\omega''\big(\overrightarrow{\Phi_\omega}\big)\vec \xi,\vec \eta\Big\rangle=\Big\langle S_\omega''\big(\overrightarrow{\Phi_\omega}\big)\vec \eta,\vec \xi\Big\rangle.
\end{align*}
Moreover, taking the derivative of $S_\omega'\big(\overrightarrow{\Phi_\omega}\big)=\vec0$ with respect to $\omega$ gives
\begin{align}\label{eq:S''Phi-partialomega}
S_\omega''\big(\overrightarrow{\Phi_\omega}\big)\partial_\omega\overrightarrow{\Phi_\omega}=-Q'\big(\overrightarrow{\Phi_\omega}\big).
\end{align}
Then a consequence of Lemma \ref{lem:partialQ} is

\begin{cor}\label{SS}
 Let $\lambda\in (-1,1)$, $|\omega|=\omega_c$; then
\begin{align*}
S_{\lambda}\big(\overrightarrow{\Phi_{\lambda}}\big)-S_{\lambda}\big(\overrightarrow{\Phi_{\omega}}\big)=o\big((\lambda-\omega)^2\big).
\end{align*}
\end{cor}

\begin{proof}
From the definition of $S_\omega(\vec u)$ in \eqref{eq:Su}, we have
\begin{align*} 
S_{\lambda}\big(\overrightarrow{\Phi_{\lambda}}\big)&-S_{\lambda}\big(\overrightarrow{\Phi_{\omega}}\big)
=S_{\omega}\big(\overrightarrow{\Phi_{\lambda}}\big)-S_{\omega}\big(\overrightarrow{\Phi_{\omega}}\big)
+(\lambda-\omega) \Big(Q\big(\overrightarrow{\Phi_{\lambda}}\big)-Q\big(\overrightarrow{\Phi_{\omega}}\big)\Big).
\end{align*}
Recall that $S_\omega'\big(\overrightarrow{\Phi_\omega}\big)=\vec0$;
then we use Taylor's expansion 
to calculate
\begin{align}
S_{\lambda}\big(\overrightarrow{\Phi_{\lambda}}\big)&-S_{\lambda}\big(\overrightarrow{\Phi_{\omega}}\big)\nonumber\\
=&
\frac12\Big\langle S_{\omega}''\big(\overrightarrow{\Phi_{\omega}}\big)\Big(\overrightarrow{\Phi_{\lambda}}-\overrightarrow{\Phi_{\omega}}\Big),
\Big(\overrightarrow{\Phi_{\lambda}}-\overrightarrow{\Phi_{\omega}}\Big)\Big\rangle\nonumber\\
&\hspace{1cm}+(\lambda-\omega) \Big(Q\big(\overrightarrow{\Phi_{\lambda}}\big)-Q\big(\overrightarrow{\Phi_{\omega}}\big)\Big)+o\big((\lambda-\omega)^2\big).\label{S-S}
\end{align}
Note that
$$
\overrightarrow{\Phi_{\lambda}}-\overrightarrow{\Phi_{\omega}}=(\lambda-\omega) \partial_\omega\overrightarrow{\Phi_\omega}+o(\lambda-\omega);
$$
then we find
\begin{align*}
\Big\langle S_{\omega}''\big(\overrightarrow{\Phi_{\omega}}\big)&\Big(\overrightarrow{\Phi_{\lambda}}-\overrightarrow{\Phi_{\omega}}\Big),
\Big(\overrightarrow{\Phi_{\lambda}}-\overrightarrow{\Phi_{\omega}}\Big)\Big\rangle\\
= &
(\lambda-\omega)^2\Big\langle S_{\omega}''\big(\overrightarrow{\Phi_{\omega}}\big)\partial_\omega\overrightarrow{\Phi_\omega},
\partial_\omega\overrightarrow{\Phi_\omega}\Big\rangle+o\big((\lambda-\omega)^2\big)\\
= &
-(\lambda-\omega)^2\Big\langle Q'\big(\overrightarrow{\Phi_\omega}\big),
\partial_\omega\overrightarrow{\Phi_\omega}\Big\rangle+o\big((\lambda-\omega)^2\big)\\
= &
-(\lambda-\omega)^2\partial_\lambda Q\big(\overrightarrow{\Phi_{\lambda}}\big)\Big|_{\lambda=\omega}+o\big((\lambda-\omega)^2\big),
\end{align*}
where we have used equality \eqref{eq:S''Phi-partialomega} in the second step.
Using Lemma \ref{lem:partialQ}, we have
$$
\partial_\lambda Q\big(\overrightarrow{\Phi_{\lambda}}\big)\Big|_{\lambda=\omega}=0.
$$
Hence,
$$
Q\big(\overrightarrow{\Phi_{\lambda}}\big)-Q\big(\overrightarrow{\Phi_{\omega}}\big)=o\big(\lambda-\omega\big),
$$
and
$$
\Big\langle S_{\omega}''\big(\overrightarrow{\Phi_{\omega}}\big)\Big(\overrightarrow{\Phi_{\lambda}}-\overrightarrow{\Phi_{\omega}}\Big),
\Big(\overrightarrow{\Phi_{\lambda}}-\overrightarrow{\Phi_{\omega}}\Big)\Big\rangle=o\big((\lambda-\omega)^2\big).
$$
Taking these two results into \eqref{S-S}, we obtain the desired estimate.
\end{proof}
\vskip 2cm

\section{Coercivity}\label{sec:coercivity}
\vskip 0.2cm

In this section, we prove a coercivity property on the Hessian of the action $S_{\omega}''\big(\overrightarrow{\Phi_\omega}\big)$.
First, we study 
the kernel of $S_{\omega}''\big(\overrightarrow{\Phi_\omega}\big)$ in the following lemma.
The proof is standard, and 
it is a consequence of the result from \cite{Weinstein-SJMA-85}.

\begin{lem}\label{lem:Ker} The kernel of $S_{\omega}''\big(\overrightarrow{\Phi_\omega}\big)$ satisfies that
\begin{align*}
Ker\Big(S_{\omega}''\big(\overrightarrow{\Phi_\omega}\big)\Big)=\big\{C\partial_x\overrightarrow{\Phi_\omega}:C\in \R \big\}. 
\end{align*}
\end{lem}

\begin{proof}
First, we need to show the relationship ``$\supset$". For any $\vec f \in \big\{C\partial_x\overrightarrow{\Phi_\omega}:C\in \R\big\}$, 
using \eqref{elliptic}, we have
\begin{align}
S_{\omega}''\big(\overrightarrow{\Phi_\omega}\big)\vec f=S_{\omega}''\big(\overrightarrow{\Phi_\omega}\big)\big(C\partial_x\overrightarrow{\Phi_\omega}\big)=C\left(\begin{array}{c}
\partial_x\big(-\partial_{xx}{\phi}_{\omega}+(1-{\omega}^2){\phi}_{\omega}-{\phi}_{\omega}^{p+1}\big)\\
-{\omega}{\phi}_{\omega}'+{\omega}{\phi}_{\omega}'\\
\end{array}\right)=\vec 0.\label{ker1}
\end{align}
Then \eqref{ker1} implies that $\vec f$ is in the kernel of $S_{\omega}''\big(\overrightarrow{\Phi_\omega}\big)$, and we have the conclusion
$$
Ker\Big(S_{\omega}''\big(\overrightarrow{\Phi_\omega}\big)\Big)\supset\big\{C\partial_x\overrightarrow{\Phi_\omega}:C\in \R \big\}. 
$$

Second, we prove the reverse relationship ``$\subset$". For any $\vec f \in Ker\Big(S_{\omega}''(\overrightarrow{\Phi_\omega})\Big)$,  by the expression of $S_{\omega}''\big(\overrightarrow{\Phi_\omega}\big)$ in \eqref{eq:S''Phi}, we have
\begin{align}\label{Kereq}
 \left\{\begin{aligned}
           -\partial_{xx}f+(1-{\omega}^2)f-(p+1){\phi}_{\omega}^pf&=0,\\
           g+{\omega}f&=0.
        \end{aligned}
 \right.
\end{align}
By the work of Weinstein \cite{Weinstein-SJMA-85}, the only solutions to (\ref{Kereq}) are
\begin{align*}
\left\{\begin{aligned}
           f&=C\partial_x{\phi}_{\omega},\\
           g&=-C{\omega}\partial_x{\phi}_{\omega},
        \end{aligned} \qquad C\in \R. 
 \right.
 \end{align*}
This implies that $\vec f \in \big\{C\partial_x\overrightarrow{\Phi_\omega}:C\in \R\big\}$, 
and we have
$$
Ker\Big(S_{\omega}''\big(\overrightarrow{\Phi_\omega}\big)\Big)\subset\big\{C\partial_x\overrightarrow{\Phi_\omega}:C\in \R \big\}. 
$$
Finally, combining the two relationship gives us
\begin{align*}
Ker\Big(S_{\omega}''\big(\overrightarrow{\Phi_\omega}\big)\Big)=\big\{C\partial_x\overrightarrow{\Phi_\omega}:C\in \R \big\}. 
\end{align*}
This gives the proof of the lemma.
\end{proof}

The second lemma is
the uniqueness of the negative eigenvalue of $S_{\omega}''\big(\overrightarrow{\Phi_\omega}\big)$.
\begin{lem}\label{lem:negative} $S_{\omega}''(\overrightarrow{\Phi_\omega})$ exists only one negative eigenvalue.
\end{lem}
\begin{proof}
It is known that the operator $-\partial_{xx}+(1-{\omega}^2)-(p+1){\phi}_{\omega}^p$ has only one negative eigenvalue (see \cite{Weinstein-SJMA-85}), and we denote it by $\lambda_{-1}$. Then there exists a unique associated eigenvector $\zeta \in H^1(\R)$ such that
\begin{align}\label{negativeeq}
-\partial_{xx}{\zeta}+(1-{\omega}^2){\zeta}-(p+1){\phi}_{\omega}^p{\zeta}=\lambda_{-1}\zeta.
\end{align}
Using the expression of $S_{\omega}''(\overrightarrow{\Phi_\omega})$ in \eqref{eq:S''Phi}, we have
\begin{align*}
\Big\langle S_{\omega}''(\overrightarrow{\Phi_\omega})&\overrightarrow{\Phi_\omega},\overrightarrow{\Phi_\omega}\Big\rangle\\
=&\int_\R (-\partial_{xx}\phi_\omega+\phi_\omega-(p+1){\phi_\omega}^{p+1}-\omega^2 \phi_\omega,-\omega\phi_\omega+\omega\phi_\omega)
\cdot \left(\begin{array}{c}
{\phi_\omega}\\
-{\omega}{\phi_\omega}
\end{array}\right)\dx\\
=&-p\|{\phi}_{\omega}\|_{L^{p+2}}^{p+2}<0.
\end{align*}
This implies that $S_{\omega}''(\overrightarrow{\Phi_\omega})$ has at least one negative eigenvalue, say, $\mu_0$.
Assume its associated eigenvector ${\vec \eta_0}=(\xi_0,\eta_0)^T$, that is,
\begin{align*}
S_{\omega}''(\overrightarrow{\Phi_\omega})\vec \eta_0=\mu_0 \vec \eta_0.
\end{align*}
Using the expression of $S_{\omega}''(\overrightarrow{\Phi_\omega})$ in \eqref{eq:S''Phi} again, the last equality yields
\begin{align*}
 \left\{\begin{aligned}
           -\partial_{xx}{\xi_0}+{\xi_0}-(p+1){\phi}_{\omega}^p\xi_0+{\omega}{\eta_0}&=\mu_0{\xi_0},\\
           {\eta_0}+{\omega}{\xi_0}&=\mu_0{\eta_0}.
        \end{aligned}
 \right.
\end{align*}
From the second equality, we have ${\eta_0}=-\frac{\omega}{1-\mu_0}\xi_0$. Then we substitute it into the first equality to get
\begin{align*}
           -\partial_{xx}{\xi_0}+(1-{\omega}^2){\xi_0}-(p+1){\phi}_{\omega}^p{\xi_0}&=\mu_0\Big(\frac{{\omega}^2}{1-\mu_0}+1\Big)\xi_0.
\end{align*}
Hence, by \eqref{negativeeq}, $(\mu_0,\vec \eta_0)$ is exactly the  pair satisfying
\begin{align}\label{lambda-lambda0}
\mu_0=\frac 12\Bigg(\lambda_{-1}+\omega^2+1-\sqrt{\lambda_{-1}^2+2(\omega^2-1)\lambda_{-1}+(\omega^2+1)^2}\Bigg),\qquad   \vec \eta_0=\begin{pmatrix} \zeta\\ \displaystyle\frac{\omega \zeta}{\mu_0-1} \end{pmatrix}.
\end{align}
This implies that $S_{\omega}''(\overrightarrow{\Phi_\omega})$ has exactly one simple negative eigenvalue.
This completes the proof of Lemma \ref{lem:negative}.
\end{proof}

The next lemma gives one of the negative direction of $S_{\omega}''\big(\overrightarrow{\Phi_\omega}\big)$.
\begin{lem}\label{lem:S''=Psi} Let
\begin{align*}
\vec \psi_\omega=\frac{1}{2\omega}\left(\begin{array}{c}
\partial_\omega \phi_\omega\\
-\omega\partial_\omega\phi_\omega
\end{array}\right), \qquad\overrightarrow{\Psi_\omega}=\left(\begin{array}{c}
\phi_\omega\\
0
\end{array}\right).
\end{align*}
Then
\begin{align}\label{S''Psi}
S_{\omega}''\big(\overrightarrow{\Phi_\omega}\big)\vec \psi_\omega=\overrightarrow{\Psi_\omega}.
\end{align}
Moreover,  we have
\begin{align*}
\Big\langle S_{\omega}''(\overrightarrow{\Phi_\omega})\vec \psi_\omega,\vec \psi_\omega\Big\rangle<0.
\end{align*}
\end{lem}

\begin{proof}
Taking the derivative of \eqref{elliptic} with respect to $\omega$,  
we have
\begin{align}
-\partial_{xx}(\partial_\omega\phi_\omega)+(1-\omega^2)\partial_\omega\phi_\omega-
(p+1)\phi_\omega^p\partial_\omega\phi_\omega=2\omega\phi_\omega.\label{partial-omega}
\end{align}
Using the expression of $S_{\omega}''\big(\overrightarrow{\Phi_\omega}\big)$ in \eqref{eq:S''Phi}, we have
\begin{align*}
S_{\omega}''\big(\overrightarrow{\Phi_\omega}\big)\vec \psi_\omega&=
\frac{1}{2\omega}\left(\begin{array}{c}
-\partial_{xx}(\partial_\omega\phi_\omega)+(1-\omega^2)
\partial_\omega\phi_\omega-(p+1)\phi_\omega^p\partial_\omega\phi_\omega\\
0
\end{array}\right).
\end{align*}
This combined with \eqref{partial-omega} gives
\begin{align}
S_{\omega}''\big(\overrightarrow{\Phi_\omega}\big)\vec \psi_\omega=
\left(\begin{array}{c}
\phi_\omega\\
0
\end{array}\right)=\overrightarrow{\Psi_\omega}.\label{S''Psi1}
\end{align}

Now we show $\Big\langle S_{\omega}''\big(\overrightarrow{\Phi_\omega}\big)\vec \psi_\omega,\vec \psi_\omega\Big\rangle<0$. From \eqref{S''Psi1}, we have
\begin{align}
\Big\langle S_{\omega}''\big(\overrightarrow{\Phi_\omega}\big)\vec \psi_\omega,\vec \psi_\omega\Big\rangle &=\Big\langle\overrightarrow{\Psi_\omega},\vec \psi_\omega\Big\rangle
=\int_\R (\phi_\omega,0)
\cdot \frac{1}{2\omega}\left(\begin{array}{c}
\partial_\omega \phi_\omega\\
-\omega\partial_\omega\phi_\omega
\end{array}\right)\dx\nonumber\\
&=\frac{1}{2\omega}\int_\R \phi_\omega\,\partial_\omega\phi_\omega\dx
=\frac{1}{4\omega}\partial_\omega\|\phi_\omega\|_{L^2}^2.\label{S''-psi-psi}
\end{align}
Note that, by \eqref{rescaling},
\begin{align*}
  \|\phi_\omega\|_{L^2}^2=(1-\omega^2)^{\frac 2p-\frac 12}\|\phi_0\|_{L^2}^2.
\end{align*}
Hence,
\begin{align*}
\partial_\omega \|\phi_\omega\|_{L^2}^2=-\big(\frac 4p-1 \big)\frac{\omega}{1-\omega^2}\|\phi_\omega\|_{L^2}^2<0.
\end{align*}
This completes the proof.
\end{proof}

Now we prove the following coercivity property.

\begin{prop}\label{prop:orth1}
Let $|\omega|<1$. Suppose that $\vec \eta = (\xi,\eta)^T \in H^1(\R)\times L^2(\R)$ satisfies
\begin{align}
\Big\langle\vec \eta, \partial_x\overrightarrow{\Phi_\omega}\Big\rangle
=\Big\langle\vec \eta ,\overrightarrow{\Psi_\omega}\Big\rangle=0,\label{ortheta}
\end{align}
where 
$\overrightarrow{\Psi_\omega}=(\phi_\omega,0)^T$.
Then
$$
\Big\langle S_\omega''\big(\overrightarrow{\Phi_\omega}\big)\vec \eta, \vec \eta\Big\rangle
\gtrsim \|\vec \eta\|_{H^1\times L^2}^2.
$$
\end{prop}

\begin{proof}
From the expression of $S_\omega''\big(\overrightarrow{\Phi_\omega}\big)$ in \eqref{eq:S''Phi}, we can write $S_\omega''\big(\overrightarrow{\Phi_\omega}\big)$ as
\begin{align*}
S_\omega''\big(\overrightarrow{\Phi_\omega}\big)=L+V,
\end{align*}
where
$L=\left(\begin{matrix}
-\partial_{xx}+1 & \omega\\
\omega & 1
\end{matrix}\right)$,
and $V=\left(\begin{matrix}
-(p+1)\phi_\omega^p & 0\\
0 & 0
\end{matrix}\right)$. Hence $V$ is a compact perturbation of the self-adjoint operator $L$.

\noindent\emph{Step 1. Analyse the spectrum of} $S_\omega''\big(\overrightarrow{\Phi_\omega}\big).$

We first compute the essential spectrum of $L$. Note that for any $\vec f=(f,g)^T\in H^1(\R)\times L^2(\R)$,
\begin{align}
\langle L\vec f,\vec f \rangle&= \Big\langle \left(\begin{matrix}
-\partial_{xx}+1 & \omega\\\nonumber
\omega & 1
\end{matrix}\right)\left(\begin{array}{c}
f\\
g\\
\end{array}\right),\left(\begin{array}{c}
f\\
g\\
\end{array}\right) \Big\rangle\\\nonumber
&=\int_\mathbb R (-\partial_{xx}f+f+\omega g,\omega f+g)\cdot\left(\begin{array}{c}
f\\
g\\
\end{array}\right)\dx\\
&=\|\partial_xf\|_{L^2}^2+\|f\|_{L^2}^2+2\omega \langle f,g \rangle+\|g\|_{L^2}^2\nonumber\\
&= \|\vec f\|_{H^1\times L^2}^2+2\omega \langle f,g \rangle.\label{Lf}
\end{align}
For the term $2\omega \langle f,g \rangle$, applying H\"older's and Young's inequalities, we have
\begin{align*}
|2\omega \langle f,g \rangle| \leq |\omega| \|\vec f\|_{H^1\times L^2}^2.
\end{align*}
Taking this estimate into \eqref{Lf}, we have
\begin{align*}
\langle L\vec f,\vec f \rangle\ge (1-|\omega|)\|\vec f\|_{H^1\times L^2}^2.
\end{align*}
Since $|\omega|<1$, we get
\begin{align*}
\langle L\vec f,\vec f \rangle \gtrsim \|\vec f\|_{H^1\times L^2}^2.
\end{align*}
This means that there exists  $\delta >0$ such that  the essential spectrum of $L$ is $[\delta,+\infty)$. By Weyl's Theorem, $S_\omega''\big(\overrightarrow{\Phi_\omega}\big)$ and $L$ share the same essential spectrum. So we obtain the essential spectrum of $S_\omega''\big(\overrightarrow{\Phi_\omega}\big)$. Recall that we have obtained the only one negative eigenvalue $\mu_0$ of $S_\omega''\big(\overrightarrow{\Phi_\omega}\big)$ in Lemma \ref{lem:negative} and the kernel of $S_\omega''\big(\overrightarrow{\Phi_\omega}\big)$ in Lemma \ref{lem:Ker}. So the discrete spectrum of $S_\omega''\big(\overrightarrow{\Phi_\omega}\big)$ is $\mu_0$, $0$, and the essential spectrum is $[\delta,+\infty)$.

\noindent\emph{Step 2. Positivity}.

The argument here is inspired by \cite{BeGhLecoz-CPDE-14,LecozYifei-PRE}. By Lemma \ref{lem:negative}, we have the unique negative eigenvalue $\mu_0$ and eigenvector $\vec \eta_0$ of $S_\omega''\big(\overrightarrow{\Phi_\omega}\big)$. For convenience, we normalize the eigenvector $\vec \eta_0$  such that $\|\vec \eta_0\|_{L^2\times L^2}=1$. Hence, for vector $\vec\eta\in H^1(\R)\times L^2(\R)$, by the spectral decomposition theorem we can write the decomposition of $\vec\eta$ along the spectrum of $S_\omega''\big(\overrightarrow{\Phi_\omega}\big)$,
\begin{align*}
\vec\eta=a_{\eta}\vec\eta_0+b_{\eta}\partial_x\overrightarrow{\Phi_\omega}+\vec g_\eta,
\end{align*}
where $a_{\eta},$ $b_{\eta}\in \mathbb R$, $\partial_x\overrightarrow{\Phi_\omega}\in Ker\Big(S_\omega''\big(\overrightarrow{\Phi_\omega}\big)\Big)$ and $\vec g_\eta$ lies in the positive eigenspace of $S_\omega''\big(\overrightarrow{\Phi_\omega}\big)$, that is, $\vec g_\eta$ satisfies
$$
\langle \vec g_\eta, \vec \eta_0\rangle=\langle \vec g_\eta, \partial_x\overrightarrow{\Phi_\omega}\rangle=0,
$$
and there exists an absolute constant $\sigma>0$ such that
\begin{align}
\Big\langle S_\omega''\big(\overrightarrow{\Phi_\omega}\big)\vec g_\eta, \vec g_\eta\Big\rangle
\ge \sigma \|\vec g_\eta\|_{L^2\times L^2}^2.\label{geta}
\end{align}
Since $\vec \eta$ satisfies the orthogonality condition $\Big\langle\vec \eta, \partial_x\overrightarrow{\Phi_\omega}\Big\rangle=0$ in \eqref{ortheta}  and $\big\langle \vec \eta_0, \partial_x\overrightarrow{\Phi_\omega}\big\rangle=0$, we have $b_\eta=0$, and thus
\begin{align}
\vec\eta=a_\eta\vec\eta_0+\vec g_\eta.\label{expressioneta}
\end{align}
Substituting \eqref{expressioneta} into $\Big\langle S_\omega''\big(\overrightarrow{\Phi_\omega}\big)\vec \eta, \vec \eta\Big\rangle$, we get
\begin{align*}
\Big\langle S_\omega''\big(\overrightarrow{\Phi_\omega}\big)\vec \eta, \vec \eta\Big\rangle&=\Big\langle S_\omega''\big(\overrightarrow{\Phi_\omega}\big)
(a_\eta\vec\eta_0+\vec g_\eta),a_\eta\vec\eta_0+\vec g_\eta\Big\rangle\\
&=a_\eta^2\Big\langle S_\omega''\big(\overrightarrow{\Phi_\omega}\big)\vec\eta_0,\vec\eta_0\Big\rangle
+2\mu_0 a_\eta\Big\langle \vec g_\eta,\vec\eta_0\Big\rangle
+\Big\langle S_\omega''\big(\overrightarrow{\Phi_\omega}\big)\vec g_{\eta},\vec g_{\eta}\Big\rangle.
\end{align*}
Due to the orthogonality property of eigenvector $\langle \vec g_\eta,\vec\eta_0\rangle=0$, we have
\begin{align}
\Big\langle S_\omega''\big(\overrightarrow{\Phi_\omega}\big)\vec \eta, \vec \eta\Big\rangle
&=a_\eta^2\Big\langle S_\omega''\big(\overrightarrow{\Phi_\omega}\big)\vec\eta_0,\vec\eta_0\Big\rangle
+\Big\langle S_\omega''\big(\overrightarrow{\Phi_\omega}\big)\vec g_{\eta},\vec g_{\eta}\Big\rangle \nonumber\\
&=\mu_0a_\eta^2\langle \vec \eta_0 , \vec \eta_0\rangle+\Big\langle S_\omega''\big(\overrightarrow{\Phi_\omega}\big)\vec g_{\eta},\vec g_{\eta}\Big\rangle \nonumber \\
&=\mu_0a_\eta^2+\Big\langle S_\omega''\big(\overrightarrow{\Phi_\omega}\big)\vec g_{\eta},\vec g_{\eta}\Big\rangle.\label{S''eta}
\end{align}
To $\vec\psi_\omega$, by spectral decomposition theorem again, we may write
\begin{align*}
\vec\psi_\omega=a\vec \eta_0+b\partial_x\overrightarrow{\Phi_\omega}+\vec g,
\end{align*}
where $a,b\in \R,$ and $\vec g$ lies in the positive eigenspace of $S_\omega''\big(\overrightarrow{\Phi_\omega}\big)$.
We note that $\Big\langle\vec\psi_\omega, \partial_x\overrightarrow{\Phi_\omega}\Big\rangle=0$.
Indeed, since $\phi_\omega$ is an even function, we have that  $\partial_\omega\phi_\omega$ is even and $\partial_x\phi_\omega$ is odd. Hence, we get
\begin{align*}
&\Big\langle\vec\psi_\omega, \partial_x\overrightarrow{\Phi_\omega}\Big\rangle
=\frac{1+\omega^2}{2\omega}\int_\mathbb R \partial_\omega\phi_\omega\>\partial_x\phi_\omega\,dx
=0.
\end{align*}
Then $b=0$, and thus
\begin{align*}
\vec\psi_\omega=a\vec \eta_0+\vec g.
\end{align*}
Therefore, a similar computation as above shows that
\begin{align*}
\Big\langle S_\omega''\big(\overrightarrow{\Phi_\omega}\big)\vec\psi_\omega, \vec\psi_\omega\Big\rangle
&=\Big\langle S_\omega''\big(\overrightarrow{\Phi_\omega}\big)(a\vec \eta_0+\vec g), a\vec \eta_0+\vec g\Big\rangle\\
&=\Big\langle S_\omega''\big(\overrightarrow{\Phi_\omega}\big)(a\vec\eta_0), a\vec\eta_0\Big\rangle
+\Big\langle S_\omega''\big(\overrightarrow{\Phi_\omega}\big)\vec g, \vec g\Big\rangle\\
&=\mu_0a^2 +\Big\langle S_\omega''\big(\overrightarrow{\Phi_\omega}\big)\vec g, \vec g\Big\rangle.
\end{align*}
For convenience, let $-\delta_0=\Big\langle S_\omega''\big(\overrightarrow{\Phi_\omega}\big)\vec\psi_\omega, \vec\psi_\omega\Big\rangle$. Then by Lemma \ref{lem:S''=Psi}, we know that $\delta_0 >0$. Moreover, we have
\begin{align}
-\delta_0=\mu_0 a^2 +\Big\langle S_\omega''\big(\overrightarrow{\Phi_\omega}\big)\vec g, \vec g\Big\rangle.\label{det0}
\end{align}
Using the orthogonality assumption $\Big\langle\vec \eta ,\overrightarrow{\Psi_\omega}\Big\rangle=0$ in (\ref{ortheta}) and \eqref{S''Psi}, we have
\begin{align*}
0&=\Big\langle\vec\eta,\overrightarrow{\Psi_\omega}\Big\rangle=\Big\langle a_\eta\vec\eta_0+\vec g_\eta,S_\omega''\big(\overrightarrow{\Phi_\omega}\big)\vec\psi_\omega\Big\rangle\\
&=\Big\langle a_\eta\vec\eta_0+\vec g_\eta,S_\omega''\big(\overrightarrow{\Phi_\omega}\big)(a\vec \eta_0+\vec g)\Big\rangle\\
&=\Big\langle a_\eta\vec \eta_0,S_\omega''\big(\overrightarrow{\Phi_\omega}\big)(a\vec \eta_0)\Big\rangle+\Big\langle \vec g_\eta,S_\omega''\big(\overrightarrow{\Phi_\omega}\big)\vec g\Big\rangle\\
&=\mu_0aa_\eta\langle \vec\eta_0,\vec \eta_0 \rangle +\Big\langle S_\omega''\big(\overrightarrow{\Phi_\omega}\big)\vec g,\vec g_\eta\Big\rangle\\
&=\mu_0a a_\eta+\Big\langle S_\omega''\big(\overrightarrow{\Phi_\omega}\big)\vec g,\vec g_\eta\Big\rangle.
\end{align*}
So we get the equality
\begin{align*}
0=\mu_0 aa_\eta+\Big\langle S_\omega''\big(\overrightarrow{\Phi_\omega}\big)\vec g,\vec g_\eta\Big\rangle.
\end{align*}
By the Cauchy-Schwarz inequality, we have
\begin{align*}
(\mu_0aa_\eta)^2&=\Big\langle S_\omega''\big(\overrightarrow{\Phi_\omega}\big)\vec {g}, \vec g_\eta\Big\rangle^2 \nonumber \\
&\leq \Big\langle S_\omega''\big(\overrightarrow{\Phi_\omega}\big)\vec {g}, \vec {g}\Big\rangle\Big\langle S_\omega''\big(\overrightarrow{\Phi_\omega}\big)\vec g_\eta, \vec g_\eta\Big\rangle.
\end{align*}
This gives
\begin{align}\label{lambda0}
(-\mu_0 a^2 )(-\mu_0 a_\eta^2)\leq \Big\langle S_\omega''\big(\overrightarrow{\Phi_\omega}\big)\vec {g}, \vec {g}\Big\rangle\Big\langle S_\omega''\big(\overrightarrow{\Phi_\omega}\big)\vec g_\eta, \vec g_\eta\Big\rangle.
\end{align}
The last equality combining with \eqref{det0} implies that
\begin{align*}
-\mu_0a_\eta^2\leq\frac{\Big\langle S_\omega''\big(\overrightarrow{\Phi_\omega}\big)\vec {g}, \vec {g}\Big\rangle\Big\langle S_\omega''\big(\overrightarrow{\Phi_\omega}\big)\vec g_\eta, \vec g_\eta\Big\rangle}{-\mu_0 a^2}
=\frac{\Big\langle S_\omega''\big(\overrightarrow{\Phi_\omega}\big)\vec {g}, \vec {g}\Big\rangle\Big\langle S_\omega''\big(\overrightarrow{\Phi_\omega}\big)\vec g_\eta, \vec g_\eta\Big\rangle}{\Big\langle S_\omega''\big(\overrightarrow{\Phi_\omega}\big)\vec {g}, \vec {g}\Big\rangle+\delta_0},
\end{align*}
that is,
\begin{align}\label{lambda0eta}
\mu_0a_\eta^2 \ge -\frac{\Big\langle S_\omega''\big(\overrightarrow{\Phi_\omega}\big)\vec {g}, \vec {g}\Big\rangle\Big\langle S_\omega''\big(\overrightarrow{\Phi_\omega}\big)\vec g_\eta, \vec g_\eta\Big\rangle}{\Big\langle S_\omega''\big(\overrightarrow{\Phi_\omega}\big)\vec {g}, \vec {g}\Big\rangle+\delta_0}.
\end{align}
Inserting \eqref{lambda0eta} into \eqref{S''eta}, we obtain
\begin{align*}
\Big\langle S_\omega''\big(\overrightarrow{\Phi_\omega}\big)\vec \eta, \vec \eta\Big\rangle
&\ge\Big(1-\frac{\Big\langle S_\omega''\big(\overrightarrow{\Phi_\omega}\big)\vec {g}, \vec {g}\Big\rangle}{\Big\langle S_\omega''\big(\overrightarrow{\Phi_\omega}\big)\vec {g}, \vec {g}\Big\rangle+\delta_0}\Big)\Big\langle S_\omega''\big(\overrightarrow{\Phi_\omega}\big)\vec g_\eta, \vec g_\eta\Big\rangle\\
&=\frac{\delta_0}{\Big\langle S_\omega''\big(\overrightarrow{\Phi_\omega}\big)\vec {g}, \vec {g}\Big\rangle+\delta_0}
\Big\langle S_\omega''\big(\overrightarrow{\Phi_\omega}\big)\vec g_\eta, \vec g_\eta\Big\rangle.
\end{align*}
Recalling that $\vec g_\eta$ satisfies \eqref{geta}, we have
\begin{align}
\Big\langle S_\omega''\big(\overrightarrow{\Phi_\omega}\big)\vec \eta, \vec \eta\Big\rangle
\ge\frac{\delta_0\sigma}{\Big\langle S_\omega''\big(\overrightarrow{\Phi_\omega}\big)\vec {g}, \vec {g}\Big\rangle+\delta_0}\|\vec g_\eta\|_{L^2\times L^2}^2, \quad \sigma>0.\label{S omega eta L2}
\end{align}
From the expression of $\vec \eta$ in \eqref{expressioneta} and the inequality \eqref{lambda0eta}, we have
\begin{align*}
\|\vec \eta\|_{L^2\times L^2}^2&=\|a_\eta\vec\eta_0+\vec g_\eta\|_{L^2\times L^2}^2
=a_\eta^2+\|\vec g_\eta\|_{L^2\times L^2}^2\\
&\leq   -\frac{\Big\langle S_\omega''\big(\overrightarrow{\Phi_\omega}\big)\vec {g}, \vec {g}\Big\rangle}{\mu_0\delta_0}\Big\langle S_\omega''\big(\overrightarrow{\Phi_\omega}\big)\vec \eta, \vec \eta\Big\rangle+\|\vec g_\eta\|_{L^2\times L^2}^2\\
&\lesssim \Big\langle S_\omega''\big(\overrightarrow{\Phi_\omega}\big)\vec \eta,\vec \eta\Big\rangle.
\end{align*}
Therefore, this  gives
\begin{align}
\Big\langle S_\omega''\big(\overrightarrow{\Phi_\omega}\big)\vec \eta, \vec \eta\Big\rangle
\gtrsim \big\|\vec \eta\big\|_{L^2\times L^2}^2.\label{L2L2}
\end{align}

To obtain the final conclusion, we still need to estimate
$$
\Big\langle S_\omega''\big(\overrightarrow{\Phi_\omega}\big)\vec \eta, \vec \eta\Big\rangle
\gtrsim \|\vec \eta\|_{H^1\times L^2}^2.
$$
Using the expression of $S_\omega''\big(\overrightarrow{\Phi_\omega})$ in \eqref{eq:S''Phi}, we have
\begin{align*}
\langle S_\omega''\big(\overrightarrow{\Phi_\omega}\big)\vec \eta, \vec \eta\rangle
&=\int_\mathbb R (-\partial_{xx}\xi+\xi-(p+1)\phi_\omega^p\xi+\omega \eta,\eta+\omega \xi) \cdot \left(\begin{array}{c}
\xi\\
\eta
\end{array}\right) \dx\\
&\hspace{1cm}=\|\partial_x\xi\|_{L^2}^2+
\|\vec \eta \|_{L^2\times L^2}^2+2\omega\int_\mathbb R\xi\eta\dx-(p+1)\int_\mathbb R|\phi_{\omega}|^p\xi^2\dx.
\end{align*}
Thus by H\"older's and Young's inequalities and \eqref{L2L2}, we get
\begin{align}
\|\partial_x\xi\|_{L^2}^2&=\Big\langle S_\omega''\big(\overrightarrow{\Phi_\omega}\big)\vec \eta, \vec \eta\Big\rangle-2\omega\int_\mathbb R\xi\eta\dx+(p+1)\int_\mathbb R|\phi_{\omega}|^p\xi^2\dx-\|\vec \eta \|_{L^2\times L^2}^2\nonumber\\
&\leq\Big\langle S_\omega''\big(\overrightarrow{\Phi_\omega}\big)\vec \eta, \vec \eta\Big\rangle+2|\omega|\|\xi\|_{L^2}\|\eta\|_{L^2}+(p+1)\|\phi_\omega\|_{L^\infty}^p\|\xi\|_{L^2}^2\nonumber\\
&\leq\Big\langle S_\omega''\big(\overrightarrow{\Phi_\omega}\big)\vec \eta, \vec \eta\Big\rangle+\Big(|\omega|+(p+1)\|\phi_\omega\|^p_{L^\infty}\Big)\|\vec \eta \|_{L^2\times L^2}^2\nonumber\\
&\lesssim\Big\langle S_\omega''\big(\overrightarrow{\Phi_\omega}\big)\vec \eta, \vec \eta\Big\rangle+\|\vec \eta \|_{L^2\times L^2}^2\lesssim \Big\langle S_\omega''\big(\overrightarrow{\Phi_\omega}\big)\vec \eta, \vec \eta\Big\rangle.\label{partial-x-xi}
\end{align}
Therefore, together \eqref{L2L2} and \eqref{partial-x-xi}, we obtain
\begin{align*}
\|\vec \eta\|_{H^1\times L^2}^2
=\|\partial_x\xi\|_{L^2}^2+\|\vec \eta \|_{L^2\times L^2}^2\lesssim \Big\langle S_\omega''\big(\overrightarrow{\Phi_\omega}\big)\vec \eta, \vec \eta\Big\rangle.
\end{align*}
Thus we obtain the desired result.
\end{proof}

Applying Proposition \ref{prop:orth1}, we obtain the following corollary, which is the nonstandard coercivity property and one of the key ingredients in our proof. Corollary \ref{cor:orth2} shows that we can replace the element $\partial_x\overrightarrow{\Phi_\omega}$ in the orthogonal condition \eqref{orthxi} by a suitably defined vector $\overrightarrow{\Gamma_\omega}$. The new orthogonal condition $\Big\langle\vec \eta, \overrightarrow{\Gamma_\omega}\Big\rangle=0$ has an  essential effect on the estimates of the translation parameter $y$ and $\lambda$ in Section \ref{sec:dynamic}.
\begin{cor}\label{cor:orth2}
Let $|\omega|<1$. Suppose that $\vec \eta \in H^1(\R)\times L^2(\R)$ satisfies
\begin{align}\label{orthxi}
\Big\langle\vec \eta, \overrightarrow{\Gamma_\omega}\Big\rangle
=\Big\langle\vec \eta ,\overrightarrow{\Psi_\omega}\Big\rangle=0,
\end{align}
where $\overrightarrow{\Gamma_\omega}\in H^1(\mathbb R)\times L^2(\mathbb R)$ and $\partial_x\overrightarrow{\Gamma_\omega}=\overrightarrow{\Psi_\omega}=(\phi_\omega,0)^T$. Then
\begin{align}
\Big\langle S_\omega''\big(\overrightarrow{\Phi_\omega}\big)\vec \eta, \vec \eta\Big\rangle
\gtrsim \big\|\vec \eta\big\|_{H^1\times L^2}^2.\label{S''xi}
\end{align}
\end{cor}
\begin{proof}
We define
$$\vec\xi=\vec\eta+b\partial_x\overrightarrow{\Phi_\omega},\qquad\vec\xi\in H^1(\R)\times L^2(\R).$$
If we choose
$$b=-\frac{\Big\langle\vec\eta,\partial_x\overrightarrow{\Phi_\omega}\Big\rangle}{\|\partial_x\overrightarrow{\Phi_\omega}\|_{L^2\times L^2}^2},$$
then
\begin{align*}
\Big\langle \vec \xi, \partial_x\overrightarrow{\Phi_\omega}\Big\rangle=0.
\end{align*}
Moreover, by \eqref{orthxi}, we have
\begin{align}
\Big\langle\vec \xi ,\overrightarrow{\Psi_\omega}\Big\rangle =\Big\langle\vec\eta+b\partial_x\overrightarrow{\Phi_\omega},\overrightarrow{\Psi_\omega}\Big\rangle
=\Big\langle\vec \eta, \overrightarrow{\Psi_\omega}\Big\rangle+b\Big\langle \partial_x\overrightarrow{\Phi_\omega},\overrightarrow{\Psi_\omega}\Big\rangle. \label{etaPsi}
\end{align}
Note that
\begin{align*}
b\Big\langle \partial_x\overrightarrow{\Phi_\omega},\overrightarrow{\Psi_\omega}\Big\rangle
=b\int_\mathbb R \Big(\partial_x\phi_\omega, (-\omega)\partial_x\phi_\omega \Big)\cdot \left(\begin{array}{c}
\phi_\omega\\
0
\end{array}\right)\dx=b \int_\mathbb R \partial_x\phi_\omega \phi_\omega\dx=0.
\end{align*}
Hence, $\Big\langle\vec \xi ,\overrightarrow{\Psi_\omega}\Big\rangle=0$. Therefore, $\vec \xi$ satisfies the orthogonality condition \eqref{ortheta} in Proposition \ref{prop:orth1}.
Then using the conclusion of Proposition \ref{prop:orth1} and $S_\omega''(\overrightarrow{\Phi_\omega})\partial_x \overrightarrow{\Phi_\omega}=\vec 0$, we get
\begin{align*}
\Big\langle S_\omega''(\overrightarrow{\Phi_\omega})\vec \eta, \vec \eta\Big\rangle &=\Big\langle S_\omega''(\overrightarrow{\Phi_\omega})\Big(\vec \xi-b\partial_x\overrightarrow{\Phi_\omega}\Big),\Big(\vec \xi-b\partial_x\overrightarrow{\Phi_\omega}\Big)\Big\rangle\\
&=\Big\langle S_\omega''(\overrightarrow{\Phi_\omega})\vec \xi, \vec \xi\Big\rangle-2b\Big\langle S_\omega''(\overrightarrow{\Phi_\omega}) \partial_x\overrightarrow{\Phi_\omega}, \vec \xi\Big\rangle+b^2\Big\langle S_\omega''(\overrightarrow{\Phi_\omega}) \partial_x\overrightarrow{\Phi_\omega},\partial_x\overrightarrow{\Phi_\omega}\Big\rangle\\
&=\Big\langle S_\omega''(\overrightarrow{\Phi_\omega})\vec \xi, \vec \xi\Big\rangle\gtrsim \|\vec \xi\|_{H^1\times L^2}^2,
\end{align*}
where we have used the self-adjoint property of the operator $S_\omega''(\overrightarrow{\Phi_\omega}) $ in the second step.

Now we claim that $  \|\vec \xi\|_{H^1\times L^2}^2   \gtrsim  \|\vec \eta\|_{H^1\times L^2}^2$. Indeed, using the orthogonality assumption \eqref{orthxi}, we have
\begin{align*}
\Big\langle\vec \xi, \overrightarrow{\Gamma_\omega}\Big\rangle
=\Big\langle \vec\eta+b\partial_x\overrightarrow{\Phi_\omega},\overrightarrow{\Gamma_\omega}\Big\rangle
=-b\int_\mathbb R (\phi_\omega, -\omega \phi_\omega)\cdot \left(\begin{array}{c}
\phi_\omega\\
0
\end{array}\right)
=-b\|\phi_\omega\|_{L^2}^2.
\end{align*}
Thus, by H\"older's inequality, we have
\begin{align}\label{b}
|b|=\frac{\Big|\Big\langle\vec \xi, \overrightarrow{\Gamma_\omega}\Big\rangle\Big|}{\|\phi_\omega\|_{L^2}^2}\lesssim\|\vec \xi\|_{H^1\times L^2}.
\end{align}
Now from (\ref{b}),
\begin{align*}
\|\vec \eta\|_{H^1\times L^2}&=\Big\|\vec \xi-b\partial_x\overrightarrow{\Phi_\omega}\Big\|_{H^1\times L^2}
\leq\|\vec \xi\|_{H^1\times L^2}+|b| \Big\|\partial_x\overrightarrow{\Phi_\omega}\Big\|_{H^1\times L^2}
\lesssim \|\vec \xi\|_{H^1\times L^2}.
\end{align*}
This completes the proof.
\end{proof}
\vskip 2cm

\section{Modulation}\label{sec:modulation}
\vskip 0.2cm
We now suppose for contradiction that the solitary wave solution is stable; that is, for any $\varepsilon >0$, there exists $\delta >0$ such that when
\begin{align*}
\| \vec u_0-\overrightarrow{\Phi_\omega}\|_{H^1\times L^2}<\delta,
\end{align*}
we have
\begin{align}\label{stable1}
\vec u \in U_\varepsilon \big(\overrightarrow{\Phi_\omega}\big).
\end{align}
Then the modulation theory shows that by choosing suitable parameters, the orthogonality conditions in Corollary \ref{cor:orth2} can be verified. The modulation is obtained via the standard implicit function theorem.
\begin{prop}\label{prop:modulation}
(Modulation). Let $|\omega|=\omega_c$. There exists $\varepsilon_0>0$ such that for any $\varepsilon\in(0,\varepsilon_0)$, $\vec u\in U_\varepsilon \big(\overrightarrow{\Phi_\omega}\big)$, the following properties are verified. There exist $C^1$-functions
$$
y:\R\rightarrow \R,\quad
\lambda: \R\rightarrow \R^+
$$
such that if we define $\vec\eta$ by
\begin{align}\label{change}
\vec\eta(t)=\vec u\big(t,\cdot+y(t)\big)-\overrightarrow{\Phi_{\lambda(t)}},
\end{align}
then $\vec \eta$ satisfies the following orthogonality conditions for any $t\in\R$:
\begin{align}
\Big\langle\vec\eta,\overrightarrow{\Gamma_{\lambda(t)}}\Big\rangle
=\Big\langle\vec\eta,\overrightarrow{\Psi_{\lambda(t)}}\Big\rangle=0,\label{orthlambda}
\end{align}
where $\overrightarrow{\Gamma_\lambda}\in H^1(\mathbb R)\times L^2(\mathbb R)$ and $\partial_x\overrightarrow{\Gamma_\lambda}=\overrightarrow{\Psi_\lambda}=\left(\begin{array}{c}
\phi_\lambda\\
0
\end{array}\right).$
Moreover, the following estimate verifies that
\begin{align}\label{moreover}
 \|\vec\eta\|_{H^1\times L^2}+|\lambda-\omega|\lesssim \varepsilon.
\end{align}
\end{prop}
\begin{proof}
We use the implicit function theorem to prove this proposition. Here we only give the important steps of the proof and refer the reader to \cite{Weinstein-SJMA-85,Weinstein-CPAM-86,MePa-InvenMath-04, MePa-Annal-05} for the similar argument.
Define
\begin{align*}
p=(\vec u;&\lambda,y), \qquad p_0=(\overrightarrow{\Phi_\omega};\omega,0).
\end{align*}
Let $\varepsilon$ be the parameter decided later, and define the functional pair $(F_1,F_2): U_\varepsilon \big(\overrightarrow{\Phi_\omega}\big)\times \R\times \R^+\rightarrow \R^2$ as
\begin{align*}
&F_1(p)=\Big\langle\vec \eta, \overrightarrow{\Gamma_\lambda}\Big\rangle,\quad
F_2(p)=\Big\langle\vec \eta, \overrightarrow{\Psi_\lambda}\Big\rangle.
\end{align*}
We claim that there exists $\varepsilon_0>0$, such that for any $\varepsilon\in (0,\varepsilon_0)$, there exists a unique $C^1$ map: $ U_\varepsilon \big(\overrightarrow{\Phi_\omega}\big)\rightarrow \R^+\times \R$ 
such that $(F_1(p), F_2(p))=0$.

Indeed, firstly we have
$$F_1(p_0)=F_2(p_0)=0.$$
Second, we prove that
\[
|J|=\left|\begin{array}{cc}
\partial_\lambda F_1 & \partial_y F_1\\
\partial_\lambda F_2 & \partial_y F_2
\end{array}\right|_{p=p_0}\not=0.
\]
Indeed, a direct calculation gives that
\begin{align*}
\partial_\lambda F_1(p) &
=\partial_\lambda \Big\langle\vec \eta, \overrightarrow{\Gamma_\lambda}\Big\rangle
=\partial_\lambda \Big\langle\vec u\big(t,x+y(t)\big)-\overrightarrow{\Phi_{\lambda(t)}}, \overrightarrow{\Gamma_\lambda}\Big\rangle\\
&=\Big\langle\vec u\big(t,x+y(t)\big)-\overrightarrow{\Phi_{\lambda(t)}}, \partial_\lambda\overrightarrow{\Gamma_\lambda}\Big\rangle
-\Big\langle\partial_\lambda\overrightarrow{\Phi_{\lambda(t)}}, \overrightarrow{\Gamma_\lambda}\Big\rangle.
\end{align*}
When $p=p_0$, we observe that $\vec u\big(t,x+y(t)\big)-\overrightarrow{\Phi_{\lambda(t)}}=0$, and the first term vanishes. For the second term, we note that $\overrightarrow{\Gamma_\lambda}$ is an odd vector and $\partial_\lambda\overrightarrow{\Phi_{\lambda(t)}}$ is an even vector, so we get
\begin{align*}
\partial_\lambda F_1(p)\Big|_{p=p_0}=0.
\end{align*}
A similar computation shows that
\begin{align*}
\partial_y F_1(p) \Big|_{p=p_0}&=\Big\langle\partial_x\vec u(x+y),\overrightarrow{\Gamma_\lambda}\Big\rangle\Big|_{p=p_0}=\Big\langle\partial_x\overrightarrow{\Phi_\lambda},\overrightarrow{\Gamma_\lambda}
\Big\rangle\Big|_{p=p_0}=-\|\phi_\omega\|_{L^2}^2;\\
\partial_\lambda F_2(p) \Big|_{p=p_0}&=-\Big\langle\partial_\lambda\overrightarrow{\Phi_\lambda},\overrightarrow{\Psi_\lambda}\Big\rangle\Big|_{p=p_0}=
-\Big\langle\partial_\lambda{\phi_\lambda},{\phi_\lambda}\Big\rangle\Big|_{p=p_0}=-\frac{1}{2}\partial_\lambda\|\phi_\lambda\|_{L^2}^2\Big|_{p=p_0}=
\frac{1}{2\omega}\|\phi_\omega\|_{L^2}^2;\\
\partial_y F_2(p) \Big|_{p=p_0}&=\Big\langle\partial_x\overrightarrow{\Phi_\lambda},\overrightarrow{\Psi_\lambda}
\Big\rangle\Big|_{p=p_0}
=\int_\mathbb R \partial_x \phi_\lambda \phi_\lambda \dx \Big|_{p=p_0}=0.
\end{align*}
Then we find that
\begin{align*}
\left|\begin{array}{cc}
\partial_\lambda F_1 & \partial_y F_1\\
\partial_\lambda F_2 & \partial_y F_2
\end{array}\right|_{p=p_0}=\frac{1}{2\omega}\|\phi_\omega\|_{L^2}^4\not=0.
\end{align*}
Therefore, the implicit function theorem implies that there exists $\varepsilon_0>0$
such that for any $\varepsilon\in(0,\varepsilon_0)$,
$\vec u\in U_\varepsilon \big(\overrightarrow{\Phi_\omega}\big)$, there exist unique $C^1$-functions
$$
y:U_\varepsilon \big(\overrightarrow{\Phi_\omega}\big)\rightarrow \R,\quad
\lambda: U_\varepsilon \big(\overrightarrow{\Phi_\omega}\big)\rightarrow \R^+,
$$
such that
\begin{align}\
\Big\langle\vec\eta,\overrightarrow{\Gamma_{\lambda}}\Big\rangle
=\Big\langle\vec\eta,\overrightarrow{\Psi_{\lambda}}\Big\rangle=0.
\end{align}
Furthermore,
\begin{align*}
\left(\begin{matrix}
\partial_u \lambda & \partial_v \lambda \\
\partial_u y & \partial_v y
\end{matrix}\right)=J^{-1}\left(\begin{matrix}
\partial_uF_1 & \partial_v F_1 \\
\partial_uF_2 & \partial_v F_2
\end{matrix}\right).
\end{align*}
This implies that
\begin{align*}
|\lambda-\omega|\lesssim\|\vec u-\overrightarrow{\Phi_\omega}\|_{H^1\times L^2}<\varepsilon.
\end{align*}
This finishes the proof of the proposition.
\end{proof}

\vskip 2cm
\section{Dynamic of the parameters}\label{sec:dynamic}
\vskip 0.2cm
In this section, we control the modulation parameters $y$ and $\lambda$. The effect of giving a precise control on modulation parameters is to obtain the structure of $I'(t)$ in Section \ref{sec:mainthm}. The main result is the following.
\begin{prop}\label{prop:y-lambda}
Let $\vec u=(u,v)^T$ be the solution of \eqref{eq:ut} with $\vec u\in U_\varepsilon \big(\overrightarrow{\Phi_\omega}\big)$, where $\varepsilon$ is obtained in Proposition \ref{prop:modulation}. Let $y,\lambda$, $ \vec\eta=(\xi,\eta)^T$ be the parameters and vector obtained in Proposition \ref{prop:modulation};
then
\begin{align*}
\dot y-\lambda=\|\phi_\lambda\|_{L^2}^{-2}\Big[Q\big(\overrightarrow{\Phi_\lambda}\big)-Q\big(\overrightarrow{\Phi_\omega}\big)\Big]
-\|\phi_\lambda\|_{L^2}^{-2}\Big[Q(\vec u_0)-Q\big(\overrightarrow{\Phi_\omega}\big)\Big]+O(\|\vec\eta\|_{H^1\times L^2}^2)
\end{align*}
and
\begin{align*}
\dot\lambda=O\big(\|\vec\eta\|_{H^1\times L^2}\big).
\end{align*}
\end{prop}

The proof of the proposition is split into the following two lemmas. The first lemma is
\begin{lem}\label{para}
Under the same assumption in Proposition \ref{prop:y-lambda}, we have
\begin{align*}
\dot y-\lambda=-\|\phi_\lambda\|_{L^2}^{-2}\langle\eta,\phi_\lambda\rangle+O\big(\|\vec\eta\|_{H^1\times L^2}^2\big),
\end{align*}
and
\begin{align*}
\dot\lambda=O\big(\|\vec\eta\|_{H^1\times L^2}\big).
\end{align*}
\end{lem}
\begin{proof}
Recall the definition $\vec\eta(t)=\vec u\big(t,\cdot+y(t)\big)-\overrightarrow{\Phi_{\lambda(t)}}$ in \eqref{change}, that is,
\begin{align}\label{uv}
 \left\{\begin{aligned}
            u(t,x)&=\phi_{\lambda}\big(x-y(t)\big)+\xi\big(t,x-y(t)\big),\\
            v(t,x)&=-\lambda\phi_{\lambda}\big(x-y(t)\big)+\eta\big(t,x-y(t)\big).
        \end{aligned}
 \right.
\end{align}
Using the first equation of the equivalent system \eqref{eq:ut}, we have
\begin{align}
 \dot\lambda\partial_\lambda\phi_\lambda          -(\dot y-\lambda) \partial_x \phi_\lambda=  -\dot\xi +(\dot y-\lambda)\partial_x \xi    +\lambda \partial_x\xi +\partial_x \eta.\label{substitute1}
\end{align}
We recall the definition of $\overrightarrow{\Gamma_\lambda}$ in Proposition \ref{prop:modulation} and denote $\gamma_\lambda$ as the first component of $\overrightarrow{\Gamma_\lambda}$. Now we multiply both sides of equality \eqref{substitute1} by $\gamma_\lambda$ and integrate to obtain
\begin{align}
\langle\dot\lambda\partial_\lambda\phi_\lambda,&\gamma_\lambda\rangle    -\langle(\dot y-\lambda)\partial_x\phi_\lambda,\gamma_\lambda\rangle\nonumber\\
&=\langle-\dot\xi,\gamma_\lambda\rangle   +(\dot y-\lambda)\langle\partial_x\xi,\gamma_\lambda\rangle   +\lambda\langle\partial_x\xi,\gamma_\lambda\rangle  +\langle\partial_x\eta,\gamma_\lambda\rangle.\label{simplify}
\end{align}
We know that $\phi_\lambda$ is an even function and $\gamma_\lambda$ is an odd function, so $\langle\dot\lambda\partial_\lambda\phi_\lambda,\gamma_\lambda\rangle=0$. By the orthogonality condition \eqref{orthlambda}, we have $$\langle\partial_x\xi,\gamma_\lambda\rangle=-\langle\vec\eta,\overrightarrow{\Psi_\lambda}\rangle=0,$$
so we get
$$\langle\dot\xi,\gamma_\lambda\rangle=\partial_t\langle \xi,\gamma_\lambda \rangle-\langle \xi,\partial_t\gamma_\lambda \rangle=\partial_t\langle \vec \eta,\overrightarrow{\Gamma_\lambda} \rangle-\langle \xi,\partial_t\gamma_\lambda \rangle = -\langle \xi,\partial_t\gamma_\lambda \rangle=- \dot \lambda\langle \xi,\partial_\lambda\gamma_\lambda \rangle.$$
Thus, we simplify equality \eqref{simplify} to obtain
\begin{align}
(\dot y-\lambda)\|\phi_\lambda\|_{L^2}^2     -\dot\lambda\langle\xi,\partial_\lambda\gamma_\lambda\rangle =-\langle\eta,\phi_\lambda\rangle.\label{equality1}
\end{align}
Next we multiply both sides of equality \eqref{substitute1} by the first component of $\overrightarrow{\Psi_\lambda}$ and integrate to obtain
\begin{align}
\langle\dot\lambda\partial_\lambda\phi_\lambda,&\phi_\lambda\rangle    -(\dot y-\lambda)\langle\partial_x\phi_\lambda,\phi_\lambda\rangle\nonumber\\
&=\langle-\dot\xi,\phi_\lambda\rangle   +(\dot y-\lambda)\langle\partial_x\xi,\phi_\lambda\rangle   +\langle\lambda\partial_x\xi,\phi_\lambda\rangle  +\langle\partial_x\eta,\phi_\lambda\rangle.\label{simplify2}
\end{align}
Now we consider the term in \eqref{simplify2} one by one.
From Lemma \ref{lem:partialQ}, we have $\partial_\lambda \|\phi_\lambda\|_{L^2}^2=- \frac{\|\phi_\lambda\|_{L^2}^2}{\lambda}$, so $$\langle\dot\lambda\partial_\lambda\phi_\lambda,\phi_\lambda\rangle=\dot\lambda \int_\mathbb R \phi_\lambda\partial_\lambda\phi_\lambda=\frac 12 \dot \lambda \partial_\lambda \| \phi_\lambda\|_{L^2}^2=-\frac{\dot \lambda}{2\lambda}\| \phi_\lambda\|_{L^2}^2.$$
The term $-(\dot y-\lambda)\langle\partial_x\phi_\lambda,\phi_\lambda\rangle$ vanishes as $\phi_\lambda$ is an even function. By the orthogonality condition \eqref{orthlambda}, we have
$$\langle\dot\xi,\phi_\lambda\rangle=\partial_t \langle\xi,\phi_\lambda\rangle-\langle\xi,\partial_t\phi_\lambda\rangle=\partial_t \langle\vec\eta,\overrightarrow{\Psi_{\lambda}}\rangle-\langle\xi,\partial_t\phi_\lambda\rangle=-\langle\xi,\partial_t\phi_\lambda\rangle.$$ Thus we simplify equality \eqref{simplify2} to obtain
\begin{align}
\dot\lambda\Big[-\frac{1}{2\lambda}\|\phi_\lambda\|_{L^2}^2-\langle\xi,\partial_\lambda\phi_\lambda\rangle\Big]+(\dot y-\lambda)\langle\xi,\partial_x\phi_\lambda\rangle =-\langle\lambda\xi+\eta,\partial_x\phi_\lambda\rangle.\label{equality2}
\end{align}
Since $\overrightarrow{\Psi_\lambda},\overrightarrow{\Gamma_\lambda},\overrightarrow{\Phi_\lambda}$ are smooth functions with exponential decay, combining \eqref{equality1} and \eqref{equality2}, we get
\begin{align}\label{combine}
 \left\{\begin{aligned}
&(\dot y-\lambda)\|\phi_\lambda\|_{L^2}^2     -\dot\lambda\langle\xi,\partial_\lambda\gamma_\lambda\rangle =-\langle\eta,\phi_\lambda\rangle,\\
&\dot\lambda\Big[-\frac{1}{2\lambda}\|\phi_\lambda\|_{L^2}^2-\langle\xi,\partial_\lambda\phi_\lambda\rangle\Big]+(\dot y-\lambda)\langle\xi,\partial_x\phi_\lambda\rangle =O(\|\vec\eta\|_{H^1\times L^2}).
        \end{aligned}
 \right.
\end{align}
We denote
\begin{align*}
 A=\left(\begin{matrix}
-\langle\xi, \partial_\lambda\gamma_\lambda\rangle & \|\phi_\lambda\|_{L^2}^2\\
-\frac{1}{2\lambda}\|\phi_\lambda\|_{L^2}^2-\langle\xi,\partial_\lambda\phi_\lambda\rangle & \langle\xi,\partial_x\phi_\lambda\rangle
\end{matrix}\right).
\end{align*}
Then by a direct calculation, we get
\begin{align*}
\left(\begin{array}{c}
\dot\lambda\\
\dot y-\lambda
\end{array}\right)=
A^{-1}\left(\begin{array}{c}
-\langle \eta,\phi_\lambda\rangle\\
O\big(\|\vec\eta\|_{H^1\times L^2}\big)
\end{array}\right)=
\left(\begin{array}{c}
O\big(\|\vec\eta\|_{H^1\times L^2}\big)\\
-\|\phi_\lambda\|_{L^2}^{-2}\langle\eta,\phi_\lambda\rangle+O\big(\|\vec\eta\|_{H^1\times L^2}^2\big)
\end{array}\right).
\end{align*}
This proves the lemma.
\end{proof}

The second lemma we need is the following.
\begin{lem}\label{lem:doty-lambda}
Under the same assumption in Proposition \ref{prop:y-lambda}, we have
\begin{align*}
\int_\mathbb R\eta\phi_\lambda\dx&=\Big[Q(\vec u_0)-Q\big(\overrightarrow{\Phi_\omega}\big)\Big]+\Big[Q\big(\overrightarrow{\Phi_\omega}\big)-Q\big(\overrightarrow{\Phi_\lambda}\big)\Big]
+O\big(\|\vec\eta\|_{H^1\times L^2}^2\big).
\end{align*}
\end{lem}
\begin{proof}
Using equality (\ref{uv}) and the expression $Q(\vec u)=\int_\mathbb R uv\dx$, we have
\begin{align*}
Q(\vec u)&=Q\left(\begin{array}{c}
\phi_\lambda+\xi\\
-\lambda\phi_\lambda+\eta
\end{array}\right)\\
&= \int_\mathbb R -\lambda\phi_\lambda^2\dx    -\lambda \int_\mathbb R\xi\phi_\lambda\dx   +\int_\mathbb R\eta\phi_\lambda\dx    +\int_\mathbb R \xi\eta\dx.
\end{align*}
Now we analyse the last equality one by one. By \eqref{Qphi}, we have $Q\big(\overrightarrow{\Phi_\lambda}\big)=\int_\mathbb R -\lambda\phi_\lambda^2\dx$. Recall that we have the orthogonality condition $\Big\langle\vec\eta,\overrightarrow{\Psi_{\lambda(t)}}\Big\rangle=0$ in \eqref{orthlambda}, then
$$-\lambda \int_\mathbb R\xi\phi_\lambda\dx=-\lambda \int_\mathbb R \vec\eta\cdot \overrightarrow{\Psi_{\lambda(t)}}\dx =0.$$
The final term gives $\int_\mathbb R\xi\eta\dx=O\big(\|\vec\eta\|_{H^1\times L^2}^2\big)$. Therefore,
\begin{align*}
Q(\vec u)=Q\big(\overrightarrow{\Phi_\lambda}\big)+\int_\mathbb R\eta\phi_\lambda\dx+O(\|\vec\eta\|_{H^1\times L^2}^2).
\end{align*}
From the conservation law of momentum, we know
\begin{align*}
\int_\mathbb R\eta\phi_\lambda\dx&=Q(\vec u)-Q\big(\overrightarrow{\Phi_\lambda}\big)+O\big(\|\vec\eta\|_{H^1\times L^2}^2\big)\\
&=\Big[Q(\vec u_0)-Q\big(\overrightarrow{\Phi_\omega}\big)\Big]+\Big[Q\big(\overrightarrow{\Phi_\omega}\big)-Q\big(\overrightarrow{\Phi_\lambda}\big)\Big]
+O\big(\|\vec\eta\|_{H^1\times L^2}^2\big).
\end{align*}
This proves the lemma.
\end{proof}

Now we are ready to prove Proposition \ref{prop:y-lambda}.
\begin{proof}[Proof of Proposition \ref{prop:y-lambda}]
Combining the estimates obtained in Lemmas \ref{para} and \ref{lem:doty-lambda}, we have
\begin{align*}
\dot y-\lambda&=-\|\phi_\lambda\|_{L^2}^{-2}\int_\mathbb R\eta\phi_\lambda\dx+O(\|\vec\eta\|_{H^1\times L^2}^2)\\
&=\|\phi_\lambda\|_{L^2}^{-2}\Big[Q\big(\overrightarrow{\Phi_\lambda}\big)-Q\big(\overrightarrow{\Phi_\omega}\big)\Big]
-\|\phi_\lambda\|_{L^2}^{-2}\Big[Q(\vec u_0)-Q\big(\overrightarrow{\Phi_\omega}\big)\Big]+O\big(\|\vec\eta\|_{H^1\times L^2}^2\big).
\end{align*}
This gives the proof of the proposition.
\end{proof}

\vskip 2cm
\section{Localized virial identities}\label{sec:virialid}
\vskip 0.2cm
The following lemmas are the localized virial identities. One can see \cite{LiuOhtaTod-AIPANL-07} for the details of the proof.

Let $\nu$ is a $H^2$-solution of $\partial_x\nu=u$, and
$$I_1(t)=\int_\mathbb R \nu \partial_t\nu\dx.$$
\begin{lem}\label{lem:I1}
Let $\vec u \in H^1(\R)\times L^2(\R)$ be the solution of the system \eqref{eq:ut}, then
\begin{align*}
I_1'(t)=\|v\|_{L^2}^2    -\|u\|_{L^2}^2      -\|u_x\|_{L^2}^2      +\|u\|_{L^{p+2}}^{p+2}.
\end{align*}
\end{lem}
Let
$$ I_2(t)=\int_\mathbb R \varphi\big(x-y(t)\big)uv \dx,$$
then we have the following lemma.

\begin{lem}\label{lem:I2}
Let $\varphi\in C^3(\R)$, $\vec u\in H^1(\mathbb R)\times L^2(\mathbb R)$ be the solution of \eqref{eq:ut}, then
\begin{align*}
I_2'(t)=&-\dot y \int_\R \varphi'\big(x-y(t)\big)uv\dx
-\frac 12\int_\R \varphi'\big(x-y(t)\big)\Big(3|u_x|^2+v^2+u^2-\frac{2 (p+1)}{p+2}|u|^{p+2}\Big)\dx\\
&+\frac 12\int_\R \varphi'''\big(x-y(t)\big)u^2\dx.
\end{align*}
\end{lem}
\vskip 2cm

\section{Proof of the main theorem}\label{sec:mainthm}
\vskip 0.2cm

This section is devoted to prove our main theorem.

\subsection{Virial identities}\label{subsecVI}

Let $\varphi(x)$ be a smooth cutoff function, where
\begin{equation}\label{cutoff}
\varphi(x)=\left\{
\begin{aligned}
x&,\quad |x|\leq  R ,\\
0&,\quad |x|\ge 2R,
\end{aligned}
\right.
\end{equation}
$0\leq\varphi'\leq 1$, $|\varphi'''|\lesssim \frac 1 {R^2}$ for any $x\in \R$. Moreover, we denote
$$I(t)=\Big(\frac{4}{p}-2\Big)I_1(t)+2I_2(t).$$
Then we have the following lemma.
\begin{lem}\label{lem:I't}
Let $R>0$, $y$, $\lambda$, $ \vec\eta=(\xi,\eta)^T$ be the parameters and vector obtained in Proposition \ref{prop:modulation}. Then
\begin{align}\label{I'(t)}
I'(t)=&-2\Big(\frac{4}{p}+1\Big)E(\vec u_0)-\Big(4\lambda\frac{4-p}{p}+2\lambda\Big)Q(\vec u_0)+\Big(2-2\lambda^2\frac{4-p}{p}\Big)\|\phi _\lambda\|_{L^2}^2\nonumber\\
&\hspace{1cm}-2\Big(\dot y-\lambda\Big)Q(\vec u_0)+\Big(2-2\lambda^2\frac{4-p}{p}\Big)\|\xi\|_{L^2}^2+2\frac{4-p}{p}\|\lambda\xi+\eta\|_{L^2}^2+R(\vec u),
\end{align}
where
\begin{align}
R(\vec u)=2\int_\mathbb R\Big[1-\varphi'\big(x-&y(t)\big)\Big]\Big(\dot yuv+\frac{3}{2}u_x^2+\frac{1}{2}u^2+\frac{1}{2}v^2-\frac{p+1}{p+2}|u|^{p+2}\Big)\dx\nonumber\\
         &+\int_\mathbb R \varphi'''\big(x-y(t)\big)u^2 \dx.\label{Ru}
\end{align}
\end{lem}
\begin{proof}
From Lemma \ref{lem:I2} and the conservation law of momentum, we change the form of $I_2'(t)$ as
\begin{align*}
I_2'(t)=&-\dot y \int_\mathbb R \Big[\varphi'\big(x-y(t)\big)-1+1\Big]uv\dx+\frac{1}{2}\int_\mathbb R \varphi'''\big(x-y(t)\big)u^2\dx\\
&\hspace{1cm}-\frac{1}{2}\int_\mathbb R \Big[\varphi'\big(x-y(t)\big)-1+1\Big]\Big[3|u_x|^2+v^2+u^2-\frac{2(p+1)}{p+2}|u|^{p+2}\Big]\dx\\
=&-\dot yQ(\vec u_0)-\frac{1}{2}\Big[3\|u_x\|_{L^2}^2+\|u\|_{L^2}^2+\|v\|_{L^2}^2-\frac{2(p+1)}{p+2}\|u\|_{L^{p+2}}^{p+2}\Big]+\frac{1}{2}\int_\R \varphi'''\big(x-y(t)\big)u^2 \dx\\
&\hspace{1cm}+\int_\mathbb R \Big[1-\varphi'\big(x-y(t)\big)\Big]\Big(\dot yuv+\frac{3}{2}|u_x|^2+\frac{1}{2}v^2+\frac{1}{2}u^2-\frac{p+1}{p+2}|u|^{p+2}\Big)\dx.
\end{align*}
Then a direct computation gives
\begin{align*}
I'(t)=&\Big(\frac{4}{p}-2\Big)I_1'(t)+2I_2'(t)\\
     =&-\Big(\frac{4}{p}+1\Big)\|u_x\|_{L^2}^2+\Big(\frac{4}{p}-3\Big)\|v\|_{L^2}^2+\Big(-\frac{4}{p}+1\Big)\|u\|_{L^2}^2+\frac{2(p+4)}{p(p+2)}\|u\|_{L^{p+2}}^{p+2}\\
     &\hspace{1cm}+2\int_\mathbb R \Big[1-\varphi'\big(x-y(t)\big)\Big]\Big(\dot yuv+\frac{3}{2}|u_x|^2+\frac{1}{2}v^2+\frac{1}{2}u^2-\frac{p+1}{p+2}|u|^{p+2}\Big)\dx\\
     &\hspace{2cm}+\int_\mathbb R \varphi'''\big(x-y(t)\big)u^2 \dx-2\dot yQ(\vec u_0).
\end{align*}
From the conservation law of energy, we have
$$2E(\vec u_0)=\|u_x\|_{L^2}^2+\|v\|_{L^2}^2+\|u\|_{L^2}^2-\frac{2}{p+2}\|u_x\|_{L^{p+2}}^{p+2}.$$
Then
\begin{align*}
 &-\Big(\frac{4}{p}+1\Big)\|u_x\|_{L^2}^2+\Big(\frac{4}{p}-3\Big)\|v\|_{L^2}^2+\Big(-\frac{4}{p}+1\Big)\|u\|_{L^2}^2+\frac{2(p+4)}{p(p+2)}\|u\|_{L^{p+2}}^{p+2}\\
 =&-2\Big(\frac{4}{p}+1\Big)E(\vec u_0)+\frac{2(4-p)}{p}\Big[\frac{p}{4-p}\|u\|_{L^2}^2+\|v\|_{L^2}^2\Big]\\
 =&-2\Big(\frac{4}{p}+1\Big)E(\vec u_0)+\frac{2(4-p)}{p}\Big[\lambda^2\|u\|_{L^2}^2+\|v\|_{L^2}^2\Big]+2\Big(1-\lambda^2\frac{4-p}{p}\Big)\|u\|_{L^2}^2\\
 =&-2\Big(\frac{4}{p}+1\Big)E(\vec u_0) +\frac{2(4-p)}{p}\|v+\lambda u\|_{L^2}^2-4\lambda\frac{4-p}{p}
Q(\vec u_0)+2\Big(1-\lambda^2\frac{4-p}{p}\Big)\|u\|_{L^2}^2.
\end{align*}
By orthogonality condition (\ref{orthlambda}) and using formula (\ref{uv}), we have  the following two equalities:
\begin{align*}
\|u\|_{L^2}^2&=\| \phi_\lambda\|_{L^2}^2+2\langle \phi_\lambda,\xi \rangle+\|\xi\|_{L^2}^2\\
&=\| \phi_\lambda\|_{L^2}^2+2\langle \overrightarrow{\Psi_\lambda},\vec\eta \rangle+\|\xi\|_{L^2}^2=\|\phi_\lambda\|_{L^2}^2+\|\xi\|_{L^2}^2,\\
\|v+\lambda u\|_{L^2}^2&=\|-\lambda \phi_\lambda+\eta+\lambda \phi_\lambda+\lambda \xi\|_{L^2}^2=\|\lambda \xi+\eta\|_{L^2}^2.
\end{align*}
Hence, using the equalities above, we obtain
\begin{align*}
I'(t)=&-2\Big(\frac{4}{p}+1\Big)E(\vec u_0) +\frac{2(4-p)}{p}\|v+\lambda u\|_{L^2}^2-4\lambda\frac{4-p}{p}
Q(\vec u_0)+2\Big(1-\lambda^2\frac{4-p}{p}\Big)\|u\|_{L^2}^2\\
&\hspace{1cm}+2\int_\mathbb R \Big[1-\varphi'\big(x-y(t)\big)\Big]\Big(\dot yuv+\frac{3}{2}u_x^2+\frac{1}{2}v^2+\frac{1}{2}u^2-\frac{p+1}{p+2}|u|^{p+2}\Big)\dx\\
&\hspace{2cm}+\int_\mathbb R \varphi'''\big(x-y(t)\big)u^2 \dx-2\dot yQ(\vec u_0)\\
=&-2\Big(\frac{4}{p}+1\Big)E(\vec u_0) +\frac{2(4-p)}{p}\|\eta+\lambda \xi\|_{L^2}^2-4\lambda\frac{4-p}{p}
Q(\vec u_0)\\
&\hspace{.5cm}+2\Big(1-\lambda^2\frac{4-p}{p}\Big)\big(\|\phi_\lambda\|_{L^2}^2+\|\xi\|_{L^2}^2\big)\\
&\hspace{1cm}+2\int_\mathbb R \Big[1-\varphi'\big(x-y(t)\big)\Big]\Big(\dot yuv+\frac{3}{2}u_x^2+\frac{1}{2}v^2+\frac{1}{2}u^2-\frac{p+1}{p+2}|u|^{p+2}\Big)\dx\\
&\hspace{2cm}+\int_\mathbb R \varphi'''\big(x-y(t)\big)u^2 \dx-2(\dot y-\lambda+\lambda)Q(\vec u_0)\\
=&-2\Big(\frac{4}{p}+1\Big)E(\vec u_0)-2\lambda\Big(2\frac{4-p}{p}+1\Big)Q(\vec u_0)+2\Big(1-\lambda^2\frac{4-p}{p}\Big)\|\phi _\lambda\|_{L^2}^2\nonumber\\
&\hspace{1cm}-2(\dot y-\lambda)Q(\vec u_0)+2\Big(1-\lambda^2\frac{4-p}{p}\Big)\|\xi\|_{L^2}^2+2\frac{4-p}{p}\|\lambda\xi+\eta\|_{L^2}^2+R(\vec u).
\end{align*}
This proves the lemma.
\end{proof}
Now we consider $R(\vec u)$ in \eqref{Ru}.
\begin{lem}\label{lem:Ru}
Let $R(\vec u)$ be defined in \eqref{Ru}; then
$$R(\vec u)=O(\|\vec \eta\|_{H^1\times L^2}^2+\frac{1}{R}).$$
\end{lem}
\begin{proof}
Using the definition of the cutoff function $\varphi$ in \eqref{cutoff}, we have
\begin{align*}
|R(\vec u)|
=&\Bigg|\int_{\{|x-y(t)|>R\}}2\Big[1-\varphi'\big(x-y(t)\big)\Big]\\
&\Big(\dot yuv+\frac{3}{2}|u_x|^2+\frac{1}{2}u^2+\frac{1}{2}v^2-\frac{p+1}{p+2}|u|^{p+2}\Big)\dx
         +\int_\mathbb R \varphi'''\big(x-y(t)\big)u^2 \dx\Bigg|\\
         \lesssim & \int_{\{|x-y(t)|>R\}}\Big(1+|\varphi'\big(x-y(t)\big)|\Big)\Big(|\dot y||u||v|+|u_x|^2+u^2+v^2+|u|^{p+2}\Big)\dx+\frac 1 {R^2}.
\end{align*}
By H\"{o}lder's inequality, $|\varphi'|\leq 1$, and $|\dot y|\lesssim 1$ (from Lemma \ref{para}), we have
\begin{align*}
|R(\vec u)| &\lesssim \int_{\{|x-y(t)|>R\}}\big(|u_x|^2+u^2+v^2+|u|^{p+2}\big)\dx+\frac 1{R^2}\\
 &\lesssim \int_{\{|x|>R\}}\Big[(\partial_x \phi_\lambda+\partial_x\xi)^2+(\phi_\lambda+\xi)^2+(\lambda \phi_\lambda-\eta)^2+|\phi_\lambda+\xi|^{p+2}\Big]\dx+\frac 1{R^2},
\end{align*}
where we have used equality \eqref{change} in the last step. Further, using the property of exponential decay of $\partial_x\phi_\lambda,$ we have
\begin{align*}
\int_{\{|x|>R\}}(\partial_x \phi_\lambda)^2\dx
\leq C\int_{\{|x|>R\}}e^{-C|x|}\dx
\leq\frac{C}{R}.
\end{align*}
Then  Young's inequality gives
\begin{align}
\int_{\{|x|>R\}}(&\partial_x\phi_\lambda+\partial_x\xi)^2\dx\nonumber\\
&\lesssim \int_{\{|x|>R\}}\left[(\partial_x\phi_\lambda)^2+(\partial_x\xi)^2\right]\dx\nonumber\\
&\lesssim \frac 1R+\|\partial_x\xi\|_{L^2}^2.\label{fi}
\end{align}
Using a similar method, we can prove
\begin{align}
\int_{\{|x|>R\}}(\phi_\lambda+\xi)^2\dx    \leq C(\frac 1R+\|\xi\|_{L^2}^2),\label{se}\\
\int_{\{|x|>R\}}(\lambda \phi_\lambda-\eta)^2\dx\leq C(\frac 1R+\|\eta\|_{L^2}^2)\label{th},\\
\int_{\{|x|>R\}}|\phi_\lambda+\xi|^{p+2}\dx\leq C(\frac 1R+\|\xi\|_{H^1}^2)\label{fo}.
\end{align}
Thus, we combine \eqref{fi}-\eqref{fo} to obtain
\begin{align*}
|R(\vec u)| \leq C(\frac 1R+\|\vec\eta\|_{H^1\times L^2}^2).
\end{align*}
This implies that $$\displaystyle R(\vec u)=O(\|\vec \eta\|_{H^1\times L^2}^2+\frac{1}{R}).$$
This proves the lemma.
\end{proof}

\subsection{Structure of $I'(t)$}
Our purpose is to control the difference between $u$ and the modulated solitons and the modulated scaling parameter. Note that the quantities involved in $I'(t)$ are nonconserved; the main issue is to analyse the quantities in detail. In particular,  we structure $I'(t)$ as follows.

Denote
\begin{align}
\rho(\vec {u_0})=&-2\Big(\frac{4}{p}+1\Big)\Big[E(\vec u_0)-E\big(\overrightarrow{\Phi_\omega}\big)\Big]-2\lambda\Big(2\frac{4-p}{p}+1\Big)\Big[Q(\vec u_0)-Q\big(\overrightarrow{\Phi_\omega}\big)\Big]\nonumber\\
&\hspace{1cm}+2\|\phi_\lambda\|_{L^2}^{-2}Q(\vec u_0)\Big[Q(\vec u_0)-Q\big(\overrightarrow{\Phi_\omega}\big)\Big],\label{rho}\\
h(\lambda)=& -2\Big(\frac{4}{p}+1\Big)E\big(\overrightarrow{\Phi_\omega}\big)-2\lambda\Big(2\frac{4-p}{p}+1\Big)Q\big(\overrightarrow{\Phi_\omega}\big)
+2\Big(1-\lambda^2\frac{4-p}{p}\Big)\|\phi _\lambda\|_{L^2}^2 \nonumber\\
&\hspace{1cm}-2\|\phi_\lambda\|_{L^2}^{-2}Q(\vec u_0)\Big[Q\big(\overrightarrow{\Phi_\lambda}\big)-Q\big(\overrightarrow{\Phi_\omega}\big)\Big],\label{h}\\
\widetilde R(\vec u)=&R(\vec u)+2\Big(1-\lambda^2\frac{4-p}{p}\Big)\|\xi\|_{L^2}^2+2\frac{4-p}{p}\|\lambda\xi+\eta\|_{L^2}^2\nonumber\\
&\hspace{1cm}-2Q(\vec u_0)\Bigg\{(\dot y-\lambda)-\frac{1}{\|\phi_\lambda\|_{L^2}^2}\Big[Q\big(\overrightarrow{\Phi_\lambda}\big)-Q\big(\overrightarrow{\Phi_\omega}\big)\Big]\nonumber\\
&\hspace{1cm}+\|\phi_\lambda\|_{L^2}^{-2}\Big[Q(\vec u_0)-Q\big(\overrightarrow{\Phi_\omega}\big)\Big]\Bigg\}.\label{Rtilde}
\end{align}
Now we rewrite $I'(t)$ as follows. In particular, we remark that there are no one-order terms with respect to $\vec \eta$ and $\lambda$.
\begin{lem}\label{lem:I't-rho-h-R}
\begin{align*}
I'(t)=\rho(\vec u_0)+h(\lambda)+\widetilde R(\vec u).
\end{align*}
\end{lem}
\begin{proof}
We will make a direct calculation. From \eqref{I'(t)}, we know that
\begin{align*}
I'(t)&=-2\Big(\frac{4}{p}+1\Big)E(\vec u_0)-2\lambda\Big(2\frac{4-p}{p}+1\Big)Q(\vec u_0)+2\Big(1-\lambda^2\frac{4-p}{p}\Big)\|\phi _\lambda\|_{L^2}^2\nonumber\\
&\hspace{1cm}-2(\dot y-\lambda)Q(\vec u_0)+2\Big(1-\lambda^2\frac{4-p}{p}\Big)\|\xi\|_{L^2}^2+2\frac{4-p}{p}\|\lambda\xi+\eta\|_{L^2}^2+R(\vec u)\\
&=-2\Big(\frac{4}{p}+1\Big)\Big[E(\vec u_0)-E\big(\overrightarrow{\Phi_\omega}\big)\Big]-2\lambda\Big(2\frac{4-p}{p}+1\Big)\Big[Q(\vec u_0)-Q\big(\overrightarrow{\Phi_\omega}\big)\Big]\\
&\hspace{1cm}+2\|\phi_\lambda\|_{L^2}^{-2}Q(\vec u_0)\Big[Q(\vec u_0)-Q\big(\overrightarrow{\Phi_\omega}\big)\Big]\\
&\hspace{1cm} -2\Big(\frac{4}{p}+1\Big)E\big(\overrightarrow{\Phi_\omega}\big)-2\lambda\Big(2\frac{4-p}{p}+1\Big)Q\big(\overrightarrow{\Phi_\omega}\big)
+2\Big(1-\lambda^2\frac{4-p}{p}\Big)\|\phi _\lambda\|_{L^2}^2 \nonumber\\
&\hspace{1cm}-2\|\phi_\lambda\|_{L^2}^{-2}Q(\vec u_0)\Big[Q\big(\overrightarrow{\Phi_\lambda}\big)-Q\big(\overrightarrow{\Phi_\omega}\big)\Big]\\
&\hspace{1cm}+R(\vec u)+2\Big(1-\lambda^2\frac{4-p}{p}\Big)\|\xi\|_{L^2}^2+2\frac{4-p}{p}\|\lambda\xi+\eta\|_{L^2}^2\\
&\hspace{1cm}-2Q(\vec u_0)\Bigg\{(\dot y-\lambda)-\|\phi_\lambda\|_{L^2}^{-2}\Big[Q\big(\overrightarrow{\Phi_\lambda}\big)-Q\big(\overrightarrow{\Phi_\omega}\big)\Big]
+\|\phi_\lambda\|_{L^2}^{-2}\Big[Q(\vec u_0)-Q\big(\overrightarrow{\Phi_\omega}\big)\Big]\Bigg\}\\
&=\rho(\vec u_0)+h(\lambda)+\widetilde R(\vec u).
\end{align*}
This completes the proof.
\end{proof}
By Lemma \ref{lem:Ru} and Proposition \ref{prop:y-lambda}, we obtain
\begin{align}
\widetilde R(\vec u)=O\big(\|\vec \eta\|_{H^1\times L^2}^2+\frac{1}{R}\big).\label{widetilder R}
\end{align}

\subsection{Positivity of the main parts}
The main parts of $I'(t)$, $\rho(\vec u_0)$, and $h(\lambda)$ are considered in this subsection. We shall prove their positivity in the following.
\begin{lem}\label{lem:finalestimate}
Let $\vec u_0=(1+a)\overrightarrow{\Phi_\omega}$ for some small positive constant $a$. Then
\begin{align*}
&1)\qquad \rho(\vec u_0)\geq C_1a,\quad \text{for some}\quad C_1>0;\\
&2)\qquad h(\lambda)\geq C_2(\lambda-\omega)^2+O(a(\lambda-\omega)^2)+o((\lambda-\omega)^2)\quad \text{for some}\quad C_2>0.
\end{align*}
\end{lem}
\begin{proof}
1) Recall the definition of $\rho(\vec u_0)$ in \eqref{rho}:
\begin{align}
\rho(\vec u_0)=&-2\Big(\frac{4}{p}+1\Big)\Big[E(\vec u_0)-E\big(\overrightarrow{\Phi_\omega}\big)\Big]-2\lambda\Big(2\frac{4-p}{p}+1\Big)\Big[Q(\vec u_0)-Q\big(\overrightarrow{\Phi_\omega}\big)\Big]\nonumber\\
&\hspace{1cm}+2\|\phi_\lambda\|_{L^2}^{-2}Q(\vec u_0)\Big[Q(\vec u_0)-Q\big(\overrightarrow{\Phi_\omega}\big)\Big].\label{rhouse}
\end{align}
First, by Taylor's type expansion, we have
\begin{align*}
E(\vec u_0)-E\big(\overrightarrow{\Phi_\omega}\big)&=\Big\langle E'\big(\overrightarrow{\Phi_\omega}\big), \vec u_0-\overrightarrow{\Phi_\omega}\Big\rangle+O\big(\|\vec u_0-\overrightarrow{\Phi_\omega}\|_{H^1\times L^2}^2\big)\\
&=a\Big\langle E'\big(\overrightarrow{\Phi_\omega}\big), \overrightarrow{\Phi_\omega}\Big\rangle+O(a^2).
\end{align*}
Using the expression of $E'\big(\overrightarrow{\Phi_\omega}\big)$ in (\ref{eq:E'u}), we have
\begin{align}
E(\vec u_0)-E\big(\overrightarrow{\Phi_\omega}\big)&=a\int_\mathbb R (-\partial_{xx}\phi_\omega+\phi_\omega-\phi_\omega^{p+1},-\omega \phi_\omega)\cdot
\left(\begin{array}{c}
\phi_\omega\\
-\omega \phi_\omega
\end{array}\right)\dx+O(a^2)\nonumber\\
&=a\int_\mathbb R
(-\partial_{xx}\phi_\omega+(1-\omega^2)\phi_\omega-\phi_\omega^{p+1}+\omega^2\phi_\omega,-\omega \phi_\omega)\cdot
\left(\begin{array}{c}
\phi_\omega\nonumber\\
-\omega \phi_\omega
\end{array}\right)\dx+O(a^2)\nonumber\\
&=2a\omega^2\|\phi_\omega\|_{L^2}^2+O(a^2),\label{E-E}
\end{align}
where we have used equation \eqref{elliptic} in the last step. Next, we compute the term $Q(\vec u_0)-Q\big(\overrightarrow{\Phi_\omega}\big)$ in \eqref{rhouse}:
\begin{align*}
Q(\vec u_0)-Q\big(\overrightarrow{\Phi_\omega}\big)&=\Big\langle Q'\big(\overrightarrow{\Phi_\omega}\big), \vec u_0-\overrightarrow{\Phi_\omega}\Big\rangle+O\big(\|\vec u_0-\overrightarrow{\Phi_\omega}\|_{H^1\times L^2}^2\big)\\
&=a\Big\langle Q'\big(\overrightarrow{\Phi_\omega}\big), \overrightarrow{\Phi_\omega}\Big\rangle+O(a^2).
\end{align*}
Using the expression of $Q'\big(\overrightarrow{\Phi_\omega}\big)$ in \eqref{eq:Q'u}, we have
\begin{align}\label{Q-Q}
Q(\vec u_0)-Q\big(\overrightarrow{\Phi_\omega}\big)&=a\int_\mathbb R
(-\omega \phi_\omega,
\phi_\omega)\cdot\left(\begin{array}{c}
\phi_\omega\\
-\omega \phi_\omega
\end{array}\right)\dx+O(a^2)\nonumber\\
&=-2a\omega\|\phi_\omega\|_{L^2}^2+O(a^2).
\end{align}
Then we put \eqref{E-E} and \eqref{Q-Q} into the expression of $\rho(\vec u_0$):
\begin{align}\label{rho-u0}
\rho(\vec u_0)=&-2\Big(\frac{4}{p}+1\Big)\Big[E(\vec u_0)-E\big(\overrightarrow{\Phi_\omega}\big)\Big]-2\lambda\Big(2\frac{4-p}{p}+1\Big)\Big[Q(\vec u_0)-Q\big(\overrightarrow{\Phi_\omega}\big)\Big]\nonumber\\
&\hspace{1cm}+2\|\phi_\lambda\|_{L^2}^{-2}Q(\vec u_0)\Big[Q(\vec u_0)-Q\big(\overrightarrow{\Phi_\omega}\big)\Big]\nonumber\\
=&-2\Big(\frac{4}{p}+1\Big)\Big[2a\omega^2\|\phi_\omega\|_{L^2}^2+O(a^2)\Big]-2\lambda\Big(2\frac{4-p}{p}+1\Big)\Big[-2a\omega\|\phi_\omega\|_{L^2}^2+O(a^2)\Big]\nonumber\\
&\hspace{1cm}+2\|\phi_\lambda\|_{L^2}^{-2}Q(\vec u_0)\Big[-2a\omega\|\phi_\omega\|_{L^2}^2+O(a^2)\Big]\nonumber\\
=&-4a\omega^2\Big(\frac{4}{p}+1\Big)\|\phi_\omega\|_{L^2}^2+4 a\omega\lambda  \Big(2\frac{4-p}{p}+1\Big)\|\phi_\omega\|_{L^2}^2-4a\omega Q(\vec u_0)\frac{\|\phi_\omega\|_{L^2}^2}{\|\phi_\lambda\|_{L^2}^2}\nonumber\\
&\hspace{1cm}+O(a^2).
\end{align}
For the term $4a\omega\lambda \Big(2\frac{4-p}{p}+1\Big)\|\phi_\omega\|_{L^2}^2$, we have
\begin{align}\label{use}
4a\omega\lambda\Big(2\frac{4-p}{p}+1\Big)\|\phi_\omega\|_{L^2}^2
=4a\omega^2\Big(2\frac{4-p}{p}+1\Big)\|\phi_\omega\|_{L^2}^2+O(a|\lambda-\omega|).
\end{align}
For the term $\displaystyle -4a\omega Q(\vec u_0)\frac{\|\phi_\omega\|_{L^2}^2}{\|\phi_\lambda\|_{L^2}^2},$ we use the expression $\phi_\omega(x)=(1-\omega^2)^{\frac{1}{p}}\phi_0\left(\sqrt{1-\omega^2}x\right)$ in \eqref{rescaling} and Taylor's type expansion again to calculate
\begin{align*}
-4a\omega Q(\vec u_0)\frac{\|\phi_\omega\|_{L^2}^2}{\|\phi_\lambda\|_{L^2}^2}&=-4a\omega Q(\vec u_0)\frac
{(1-\omega^2)^{\frac 2p-\frac12}\|\phi_0\|_{L^2}^2}{(1-\lambda^2)^{\frac 2p-\frac12}\|\phi_0\|_{L^2}^2}=-4a\omega Q(\vec u_0)\frac
{(1-\omega^2)^{\frac 2p-\frac12}}{(1-\lambda^2)^{\frac 2p-\frac12}}\\
&=-4a\omega Q(\vec u_0)(1-\omega^2)^{\frac 2p-\frac12}\Big[(1-\omega^2)^{\frac 12-\frac 2p}+O(|\lambda-\omega|)\Big]\\
&=-4a\omega Q(\vec u_0)+Q(\vec u_0)O(a|\lambda-\omega|).
\end{align*}
From the definition of $Q(\vec u)$ in \eqref{Momentum}, we have
\begin{align*}
Q(\vec u_0)=Q\Big((1+a)\overrightarrow{\Phi_\omega}\Big)=-\omega(1+a)^2 \|\phi_\omega\|_{L^2}^2.
\end{align*}
Combining the last two estimates, we obtain
\begin{align}
-4a\omega Q(\vec u_0)\frac{\|\phi_\omega\|_{L^2}^2}{\|\phi_\lambda\|_{L^2}^2}=4a\omega^2\|\phi_\omega\|_{L^2}^2+O(a^2)+O(a|\lambda-\omega|).\label{use2}
\end{align}
Finally we put \eqref{use} and \eqref{use2} into \eqref{rho-u0} to obtain
\begin{align*}
\rho(\vec u_0)=&-4a\omega^2\Big(\frac{4}{p}+1\Big)\|\phi_\omega\|_{L^2}^2+4a\omega^2\frac{8-p}{p}
\|\phi_\omega\|_{L^2}^2\\
&\hspace{1cm}+4a\omega^2\|\phi_\omega\|_{L^2}^2+O(a|\lambda-\omega|)+O(a^2)\\
=& 4a\omega^2 \frac{4-p}{p}\|\phi_\omega\|_{L^2}^2+O(a|\lambda-\omega|)+O(a^2).
\end{align*}
Choosing $a$ and $\varepsilon_0$ small enough, where $\varepsilon_0$ is the constant in Proposition \ref{prop:modulation}, and by \eqref{moreover}, we obtain  conclusion 1) of this lemma.

2)
Recall the definition of $h(\lambda)$ from \eqref{h}:
\begin{align}
h(\lambda)= -2\Big(\frac{4}{p}+&1\Big)E\big(\overrightarrow{\Phi_\omega}\big)-2\lambda\Big(2\frac{4-p}{p}+1\Big)Q\big(\overrightarrow{\Phi_\omega}\big)
+2\Big(1-\lambda^2\frac{4-p}{p}\Big)\|\phi _\lambda\|_{L^2}^2 \nonumber\\
&-2\|\phi_\lambda\|_{L^2}^{-2}Q(\vec u_0)\Big[Q\big(\overrightarrow{\Phi_\lambda}\big)-Q\big(\overrightarrow{\Phi_\omega}\big)\Big].\label{hlambdause}
\end{align}
First, we consider the last term and claim that
\begin{align}
-2\|\phi_\lambda&\|_{L^2}^{-2}Q(\vec u_0)\Big[Q\big(\overrightarrow{\Phi_\lambda}\big)-Q\big(\overrightarrow{\Phi_\omega}\big)\Big]\nonumber\\
&=2\omega
\Big[Q\big(\overrightarrow{\Phi_\lambda}\big)-Q\big(\overrightarrow{\Phi_\omega}\big)\Big]
+o\big((\lambda-\omega)^2\big)+O\big(a(\lambda-\omega)^2\big).\label{use3}
\end{align}
To prove \eqref{use3}, we need the following equalities, which can be obtained by Taylor's type expansion and Lemma \ref{lem:partialQ}:
\begin{align}
Q\big(\overrightarrow{\Phi_\lambda}\big)-Q\big(\overrightarrow{\Phi_\omega}\big)&=\partial_\lambda Q\big(\overrightarrow{\Phi_\lambda}\big)\Bigg|_{\lambda=\omega}(\lambda-\omega)+O((\lambda-\omega)^2)\nonumber\\
&=O((\lambda-\omega)^2),\label{Q1}\\
Q(\vec u_0)-Q\big(\overrightarrow{\Phi_\omega}\big)&=O(a),\label{Q2}\\
\|\phi_\lambda\|_{L^2}^{-2}-\|\phi_\omega\|_{L^2}^{-2}&=O(|\lambda-\omega|)\label{Q3}.
\end{align}
Using \eqref{Q1}--\eqref{Q3}, we obtain
\begin{align*}
-2\|\phi_\lambda\|_{L^2}^{-2}Q(\vec u_0)&\Big[Q\big(\overrightarrow{\Phi_\lambda}\big)-Q\big(\overrightarrow{\Phi_\omega}\big)\Big]\\
=&-2\|\phi_\omega\|_{L^2}^{-2}Q\big(\overrightarrow{\Phi_\omega}\big)
\Big[Q\big(\overrightarrow{\Phi_\lambda}\big)-Q\big(\overrightarrow{\Phi_\omega}\big)\Big]+o((\lambda-\omega)^2)+O\big(a(\lambda-\omega)^2\big).
\end{align*}
Further, from \eqref{Qphi}, we get
\begin{align*}
-2\|\phi_\omega&\|_{L^2}^{-2}Q\big(\overrightarrow{\Phi_\omega}\big)
\Big[Q\big(\overrightarrow{\Phi_\lambda}\big)-Q\big(\overrightarrow{\Phi_\omega}\big)\Big]\\
&=-2\|\phi_\omega\|_{L^2}^{-2}\cdot(-\omega\|\phi_\omega\|_{L^2}^2)\cdot \Big[Q\big(\overrightarrow{\Phi_\lambda}\big)-Q\big(\overrightarrow{\Phi_\omega}\big)\Big]\\
&=2\omega\Big[Q\big(\overrightarrow{\Phi_\lambda}\big)-Q\big(\overrightarrow{\Phi_\omega}\big)\Big].
\end{align*}
Thus, we obtain
\begin{align*}
-2\|\phi_\lambda&\|_{L^2}^{-2}Q(\vec u_0)\Big[Q\big(\overrightarrow{\Phi_\lambda}\big)-Q\big(\overrightarrow{\Phi_\omega}\big)\Big]\\
&=2\omega
\Big[Q\big(\overrightarrow{\Phi_\lambda}\big)-Q\big(\overrightarrow{\Phi_\omega}\big)\Big]
+o\big((\lambda-\omega)^2\big)+O\big(a(\lambda-\omega)^2\big).
\end{align*}
This proves \eqref{use3}.

Inserting \eqref{use3} into \eqref{hlambdause}, we get
\begin{align*}
h(\lambda)=-2&\Big(\frac{4}{p}+1\Big)E\big(\overrightarrow{\Phi_\omega}\big)-2\lambda\Big(2\frac{4-p}{p}+1\Big)Q\big(\overrightarrow{\Phi_\omega}\big)
+2\Big(1-\lambda^2\frac{4-p}{p}\Big)\|\phi _\lambda\|_{L^2}^2 \\
&+2\omega
\Big[Q\big(\overrightarrow{\Phi_\lambda}\big)-Q\big(\overrightarrow{\Phi_\omega}\big)\Big]
+o\big((\lambda-\omega)^2\big)+O\big(a(\lambda-\omega)^2\big).
\end{align*}
Let
\begin{align}\label{h1}
h_1(\lambda)=& -2\Big(\frac{4}{p}+1\Big)E\big(\overrightarrow{\Phi_\omega}\big)-2\lambda\Big(2\frac{4-p}{p}+1 \Big)Q\big(\overrightarrow{\Phi_\omega}\big)\nonumber\\
&\hspace{1cm}+2\Big(1-\lambda^2\frac{4-p}{p}\Big)\|\phi _\lambda\|_{L^2}^2
+2\omega
\Big[Q\big(\overrightarrow{\Phi_\lambda}\big)-Q\big(\overrightarrow{\Phi_\omega}\big)\Big].
\end{align}
Then
\begin{align}
h(\lambda)=h_1(\lambda)+o\big((\lambda-\omega)^2\big)+O\big(a(\lambda-\omega)^2\big).\label{hh1}
\end{align}
Now we claim that
\begin{align}
h_1(\omega)=0,\quad h_1'(\omega)=0,\quad h_1''(\omega)>0.
\label{13.22}
\end{align}
We prove the claim by the following three steps.

\emph{Step 1. $h_1(\omega)=0$.}

By the definition of $h_1(\lambda),$ we have
\begin{align*}
h_1(\omega)=& -2\Big(\frac{4}{p}+1\Big)E\big(\overrightarrow{\Phi_\omega}\big)-2\omega\Big(2\frac{4-p}{p}+1 \Big)Q\big(\overrightarrow{\Phi_\omega}\big)
+2\Big(1-\omega^2\frac{4-p}{p}\Big)\|\phi _\omega\|_{L^2}^2.
\end{align*}
By \eqref{Qphi} and $E\big(\overrightarrow{\Phi_\omega}\big)$ in \eqref{Energy}, we have
\begin{align*}
h_1(\omega)
&=-2\Big(\frac{4}{p}+1\Big)\Big(\frac 12\int_\mathbb R \big(|\partial_x\phi_\omega|^2+|\phi_\omega|^2+|-\omega\phi_\omega|^2\big)\dx-\frac{1}{p+2}\int_\mathbb R|\phi_\omega|^{p+2}\dx\Big)\\
&\hspace{1cm}-2\omega\Big(2\frac{4-p}{p}+1 \Big)\Big( \int_\mathbb R-\omega\phi_\omega^2 \dx \Big)
+2\Big(1-\omega^2\frac{4-p}{p}\Big)\|\phi _\omega\|_{L^2}^2\\
&=-\frac 8p \omega^2 \|\phi_\omega\|_{L^2}^2+2\|\phi_\omega\|_{L^2}^2=0,
\end{align*}
where we have used $\omega^2=\frac p4$ in the above computation. Therefore, we have $h_1(\omega)=0.$

\noindent\emph{Step 2. $h_1'(\omega)=0$.}

Using the expression of $h_1(\lambda)$ in \eqref{h1}, we have
\begin{align}
h_1'(\lambda)=& -2\Big(2\frac{4-p}{p}+1 \Big)Q\big(\overrightarrow{\Phi_\omega}\big)
-4\lambda\frac{4-p}{p}\|\phi _\lambda\|_{L^2}^2\nonumber\\
&\hspace{1cm}+2\Big(1-\lambda^2\frac{4-p}{p}\Big)\partial_\lambda(\|\phi _\lambda\|_{L^2}^2)
+2\omega \partial_\lambda Q\big(\overrightarrow{\Phi_\lambda}\big).\label{h1' lambda}
\end{align}
By \eqref{Qphi} and Lemma \ref{lem:partialQ}, we have
\begin{align}
h_1'(\omega)&= -2\Big(2\frac{4-p}{p}+1 \Big)Q\big(\overrightarrow{\Phi_\omega}\big)
+4\frac{4-p}{p}Q\big(\overrightarrow{\Phi_\omega}\big)+2\Big(1-\omega^2\frac{4-p}{p}\Big)\partial_\lambda(\|\phi _\lambda\|_{L^2}^2)\Big|_{\lambda=\omega}\nonumber\\
&=-2Q\big(\overrightarrow{\Phi_\omega}\big)+2\Big(1-\omega^2\frac{4-p}{p}\Big)\partial_\lambda(\|\phi _\lambda\|_{L^2}^2)\Big|_{\lambda=\omega}.\label{h1'}
\end{align}
Now we compute the term $\displaystyle\partial_\lambda(\|\phi _\lambda\|_{L^2}^2)\Big|_{\lambda=\omega}$. Note that
\begin{align*}
\partial_\lambda Q\big(\overrightarrow{\Phi_\lambda}\big)=\partial_\lambda(-\lambda \|\phi_\lambda\|_{L^2}^2)
=-\|\phi_\lambda\|_{L^2}^2-\lambda \partial_\lambda(\|\phi _\lambda\|_{L^2}^2);
\end{align*}
then Lemma \ref{lem:partialQ} gives
\begin{align}
\partial_\lambda(\|\phi _\lambda\|_{L^2}^2)\Big|_{\lambda=\omega}=-\frac 1\omega\|\phi_\omega\|_{L^2}^2.\label{use4}
\end{align}
Taking \eqref{use4} into \eqref{h1'}, we get
\begin{align*}
h_1'(\omega)&=2\omega\|\phi_\omega\|_{L^2}^2+2\Big(1-\omega^2\frac{4-p}{p}\Big)\Big(-\frac 1\omega\|\phi_\omega\|_{L^2}^2\Big)\\
&=\frac 2 {\omega}\Big(\omega^2-1+\omega^2\frac{4-p}{p}\Big)\|\phi_\omega\|_{L^2}^2\\
&=\frac 2 \omega\Big(  \frac 4p \omega^2-1   \Big)\|\phi_\omega\|_{L^2}^2=0.
\end{align*}
Thus, we prove the result $h_1'(\omega)=0.$

\noindent\emph{Step 3. $h_1''(\omega)>0$.}

Taking the derivative of \eqref{h1' lambda} with respect to $\lambda$, we have
\begin{align*}
h_1''(\lambda)=&-4\frac{4-p}{p}\|\phi _\omega\|_{L^2}^2-8\lambda\frac{4-p}{p}\partial_\lambda\big(\|\phi _\lambda\|_{L^2}^2\big)\\
&\hspace{1cm}+2\Big(1-\frac{4-p}{p}\lambda^2\Big)\partial_\lambda^2\big(\|\phi _\lambda\|_{L^2}^2\big)
+2\omega\partial_\lambda^2Q\big(\overrightarrow{\Phi_\lambda}\big).
\end{align*}
Since
\begin{align*}
\partial_\lambda^2 Q\big(\overrightarrow{\Phi_\lambda}\big)&=-\partial_\lambda^2\big(\lambda \|\phi _\lambda\|_{L^2}^2\big)\\
&=-2\partial_\lambda\big(\|\phi _\lambda\|_{L^2}^2\big)-\lambda\partial_\lambda^2\big(\|\phi _\lambda\|_{L^2}^2\big),
\end{align*}
we have
\begin{align*}
h_1''(\lambda)=&-4\frac{4-p}{p}\|\phi _\omega\|_{L^2}^2-4\Big(2\frac{4-p}{p}\lambda+\omega\Big)\partial_\lambda\big(\|\phi _\lambda\|_{L^2}^2\big)\\
&\hspace{1cm}+2\Big(1-\frac{4-p}{p}\lambda^2-\lambda\omega\Big)\partial_\lambda^2\big(\|\phi _\lambda\|_{L^2}^2\big).
\end{align*}
Hence,
\begin{align*}
h_1''(\omega)=&-4\frac{4-p}{p}\|\phi _\omega\|_{L^2}^2  -4\omega\frac{8-p}{p}\partial_\lambda\big(\|\phi _\lambda\|_{L^2}^2\big)\Big|_{\lambda=\omega}
+2\Big(1-\frac{4}{p}\omega^2\Big)\partial_\lambda^2\|\phi _\lambda\|_{L^2}^2\Big|_{\lambda=\omega}.
\end{align*}
Using \eqref{use4} and $\omega^2=\frac p4$, we have
\begin{align*}
h_1''(\omega)=&-4\frac{4-p}{p}\|\phi _\omega\|_{L^2}^2+4\frac{8-p}{p}\|\phi _\omega\|_{L^2}^2=\frac {16} p\|\phi _\omega\|_{L^2}^2>0.
\end{align*}
Thus, we prove the result $h_1''(\omega)>0$. This proves the claim \eqref{13.22}.

Using \eqref{13.22} and Taylor's type extension, we get
\begin{align*}
h_1(\lambda)&=h_1(\omega)+h_1'(\omega)(\lambda-\omega)+\frac 12 h_1''(\omega)(\lambda-\omega)^2+o\big((\lambda-\omega)^2\big)\\
&\ge C_2(\lambda-\omega)^2+o(\lambda-\omega)^2,
\end{align*}
where $C_2= \frac 12 h_1''(\omega)>0$. Putting this into \eqref{hh1}, we obtain the conclusion 2) of this lemma.
\end{proof}
Hence, combining  Lemmas \ref{lem:I't-rho-h-R} and \ref{lem:finalestimate}, and \eqref{widetilder R},  we have
\begin{align}
I'(t)\geq C_1a+C_2(\lambda-\omega)^2+O\Big(\|\vec\eta\|_{{H^1}\times L^2}^2+a(\lambda-\omega)^2+\frac{1}{R}\Big).\label{estimateI't}
\end{align}
\subsection{Upper control of $\|\vec \eta\|_{H^1\times L^2}$}
From \eqref{estimateI't}, to prove the monotonicity of $I'(t)$, we only need to estimate $\|\vec \eta\|_{H^1\times L^2}$. In this subsection, we give the following estimate on $\|\vec \eta\|_{H^1\times L^2}$.
\begin{lem}\label{lem:uppercontrol}
 Let $\vec \eta$ be defined in (\ref{change}); then
\begin{align*}
\|\vec \eta\|_{H^1\times L^2}^2&
\lesssim O(a|\lambda-\omega|+a^2)+o\big((\lambda-\omega)^2\big).
\end{align*}
\end{lem}
\begin{proof}
First, since $\vec u=\Big(\overrightarrow{\Phi_\lambda}+\vec \eta\Big)(x-y)$ in \eqref{uv}, we have
\begin{align*}
S_\lambda(\vec u)-S_\lambda\big(\overrightarrow{\Phi_\lambda}\big)=&\Big\langle S_\lambda'\big(\overrightarrow{\Phi_\lambda}\big),\vec\eta\Big\rangle+\frac{1}{2}\Big\langle S_\lambda''\big(\overrightarrow{\Phi_\lambda}\big)\vec\eta,\vec\eta\Big\rangle+o(\|\vec\eta\|_{H^1\times L^2}^2).
\end{align*}
Using $S_\omega'\big(\overrightarrow{\Phi_\omega}\big)=\vec0$ and Taylor's type extension, we have
\begin{align*}
S_\lambda(\vec u)-S_\lambda\big(\overrightarrow{\Phi_\lambda}\big)=&\frac{1}{2}\Big\langle S_\lambda''\big(\overrightarrow{\Phi_\lambda}\big)\vec\eta,\vec\eta\Big\rangle+o(\|\vec\eta\|_{H^1\times L^2}^2).
\end{align*}
Then by the estimate \eqref{S''xi} in Corollary \ref{cor:orth2}, we get
\begin{align*}
S_\lambda(\vec u)-S_\lambda\big(\overrightarrow{\Phi_\lambda}\big)\gtrsim\|\vec\eta\|_{H^1\times L^2}^2.
\end{align*}
Second, note that
\begin{align*}
S_\lambda(\vec u)-S_\lambda\big(\overrightarrow{\Phi_\lambda}\big)
=S_\lambda(\vec u_0)-S_\lambda\big(\overrightarrow{\Phi_\omega}\big)
+S_\lambda\big(\overrightarrow{\Phi_\omega}\big)-S_\lambda\big(\overrightarrow{\Phi_\lambda}\big),
\end{align*}
and Taylor's type extension gives
\begin{align*}
S_\lambda(\vec u_0)-S_\lambda\big(\overrightarrow{\Phi_\omega}\big)
&=E(\vec u_0)-E\big(\overrightarrow{\Phi_\omega}\big)+\lambda\Big(Q(\vec u_0)-Q\big(\overrightarrow{\Phi_\omega}\big)\Big)\\
&=S_\omega(\vec u_0)-S_\omega\big(\overrightarrow{\Phi_\omega}\big)+(\lambda-\omega)\Big(Q(\vec u_0)-Q\big(\overrightarrow{\Phi_\omega}\big)\Big)\\
&=\Big\langle S_\omega'\big(\overrightarrow{\Phi_\omega}\big),\vec u_0-\overrightarrow{\Phi_\omega}\Big\rangle+O\Big(\|\vec u_0-\overrightarrow{\Phi_\omega}\|_{H^1\times L^2}^2\Big)+(\lambda-\omega)O\Big(\|\vec u_0-\overrightarrow{\Phi_\omega}\|_{H^1\times L^2}\Big)\\
&=O(a^2+a|\lambda-\omega|).
\end{align*}
By Corollary \ref{SS}, we have
\begin{align*}
S_\lambda\big(\overrightarrow{\Phi_\omega}\big)-S_\lambda\big(\overrightarrow{\Phi_\lambda}\big)=o((\lambda-\omega)^2).
\end{align*}
Finally, we get the desired result:
\begin{align*}
\|\vec\eta\|_{H^1\times L^2}^2&\lesssim S_\lambda(\vec u)-S_\lambda\big(\overrightarrow{\Phi_\lambda}\big)
=S_\lambda(\vec u_0)-S_\lambda\big(\overrightarrow{\Phi_\omega}\big)
+S_\lambda\big(\overrightarrow{\Phi_\omega}\big)-S_\lambda\big(\overrightarrow{\Phi_\lambda}\big)\\
&=O(a|\lambda-\omega|+a^2)+o\big((\lambda-\omega)^2\big).
\end{align*}
This completes the proof.
\end{proof}

\subsection{Proof of Theorem \ref{thm:main}}
As in the discussion above, we assume that $\vec u\in U_\varepsilon\big(\overrightarrow{\Phi_\omega}\big)$, and thus $|\lambda-\omega|\lesssim \varepsilon$. First, we note that from the definition of $I(t)$ and Young's inequality, we have the time uniform boundedness of $I(t)$:
\begin{align}
\sup\limits_{t\in\R} I(t)
\lesssim R\Big(   \|\overrightarrow{\Phi_\omega}\|_{H^1\times L^2}^2+1      \Big).\label{supI}
\end{align}
Now we estimate on $I'(t)$. From \eqref{estimateI't} and Lemma \ref{lem:uppercontrol},
\begin{align*}
I'(t)&\ge C_1a+C_2(\lambda-\omega)^2+O(\|\vec \eta\|_{H^1\times L^2}^2)+O\Big(a(\lambda-\omega)^2+\frac{1}{R}\Big)\\
&\ge \frac{1}{2}C_1a+C_2(\lambda-\omega)^2+O(a|\lambda-\omega|+a^2)+o\big((\lambda-\omega)^2\big)+O\Big(\frac 1R\Big).
\end{align*}
By \eqref{moreover}, choosing R satisfying $\frac 1R\leq a^2,$ and choosing $\varepsilon$, $a_0$ small enough, we obtain that for any $a\in (0,a_0)$,
\begin{align*}
  I'(t)&\ge \frac{1}{2}C_1a+C_2(\lambda-\omega)^2+O(a^2+a|\lambda-\omega|)+o(\lambda-\omega)^2\\
  &\ge\frac14C_1a+\frac 12C_2(\lambda-\omega)^2.
\end{align*}
This implies $I(t)\to +\infty$ when $t\to +\infty$, which is contradicted with (\ref{supI}).
Hence we prove the instability of the solitary wave $\phi_\omega(x-\omega t)$ and thus give the proof of Theorem \ref{thm:main}.

\section*{Acknowledgements}
The authors are grateful to the associate editor who gave useful notes on grammatical/typographical errors  and to the anonymous referees who carefully read the paper and gave many helpful comments and suggestions.

\end{document}